\documentclass[12pt]{article}
\usepackage{setspace}

\usepackage{amsmath, amscd}
\usepackage{amsfonts}
\usepackage{amssymb}
\usepackage{amsthm}
\usepackage{upref}
\usepackage{multirow}
\usepackage{graphicx}
\usepackage{subfig}
\usepackage{url}
\usepackage{mathtools}
\usepackage{pinlabel}
\usepackage{color}
\usepackage{hyperref}
\usepackage[capitalize,nameinlink,noabbrev]{cleveref}
\usepackage{changepage}
\usepackage{epstopdf}
\usepackage{tikz}
\usepackage{enumerate}
\usepackage{multicol}
\theoremstyle{plain}
\newtheorem{theorem}{Theorem}
\newtheorem{lemma}{Lemma}
\newtheorem{proposition}{Proposition}

\numberwithin{theorem}{section}
\numberwithin{lemma}{section}
\numberwithin{proposition}{section}
\numberwithin{corollary}{section}

\theoremstyle{definition}
\newtheorem{definition}{Definition}
\newtheorem{remark}{Remark}
\newtheorem{example}{Example}
\numberwithin{definition}{section}
\numberwithin{remark}{section}
\numberwithin{example}{section}
\newtheorem{fact}{Fact}

\newcommand{\Aut}{{\mathrm{Aut}}}
\newcommand{\Out}{{\mathrm{Out}}}
\newcommand{\Inn}{{\mathrm{Inn}}}
\newcommand{\N}{{\mathbb{N}}}
\newcommand{\Z}{{\mathbb{Z}}}

\newcommand{\id}{{\mbox{id}}}

\begin{document}
\title{A characterization of virtually cyclic outer automorphism groups of right-angled Coxeter groups}
\author{Christina Angharad Hodges}
\date{\textbf{April 29, 2026}}
\maketitle

Existing research gives conditions for when the outer automorphism group of a graph product of primary cyclic groups $W_\Gamma$ is finite, virtually abelian, or large. We seek to prove a set of conditions for when this outer automorphism group is virtually cyclic. To this end, we study the finite index subgroup $\Out^0(W_\Gamma)$, which is generated by specific partial conjugations. The presence or absence of Coxeter and non-Coxeter separating intersections of links (SILs), separating triple intersections of links (STILs), and flexible separating intersections of links (FSILs) in $\Gamma$ determines algebraic properties of $\Out^0(W_\Gamma)$. We identify each SIL with a pair of partial conjugations in $\Out^0(W_\Gamma)$ and place restrictions on the SILs in $\Gamma$ to ensure that $\Out^0(W_\Gamma)$ is virtually $\Z$ both when $\Gamma$ is connected or disconnected. In particular, this applies to the study of right-angled Coxeter groups. 

\section{Introduction}

\subsection{Motivation and related efforts\label{Section1.1}}

A graph product of groups is a group construction introduced by \cite{wreo236} which interpolates between the direct product and free product of a collection of groups. This paper studies graph products of primary cyclic groups, denoted $W_\Gamma$, which notably include right-angled Coxeter groups. For a graph product of primary cyclic groups, the inner and outer automorphism groups (denoted $\Inn(W_\Gamma)$ and $\Out(W_\Gamma)$, respectively) constitute all but a finite part of the automorphism group $\Aut(W_\Gamma)$. Thus a study of the outer automorphism group leads to a deeper understanding of the entire automorphism group. Our results do not generalize to graph products of directly indecomposable groups, which include right-angled Artin groups. \cite{charney2008automorphismshigherdimensionalrightangledartin} and \cite{Guirardel_2018}, among others, study the structure of the outer automorphism group of a right-angled Artin group. 

We build upon the results of \cite{charney2009automorphismgroupgraphproduct} and \cite{Sale_2021} on $\Out(W_\Gamma)$, and the results of \cite{Laurence_1995}, \cite{Gutierrez2008}, and \cite{corredor2009generatingsetautomorphismgroup} on $\Aut(W_\Gamma)$.  Our results are sandwiched between the work of Gutierrez, Piggott, and Ruane (\cite{Gutierrez_2012}) and that of Susse and Sale \cite{sale2017outerautomorphismgroupsrightangled}. \cite{Gutierrez_2012} characterized when $\Out(W_\Gamma)$ is finite, depending on the presence of a separating intersection of links (SIL). In \cite{sale2017outerautomorphismgroupsrightangled}, Susse and Sale furthered this work by characterizing when $\Out(W_\Gamma)$ is virtually abelian or large (that is, it maps onto a non-abelian free group). This paper aims to fill the gap between these two papers by describing which properties of $\Gamma$ occur exactly when $\Out(W_\Gamma)$ is virtually cyclic, furthering the study of how the number of SIL's in $\Gamma$ impacts structural properties of $\Out(W_\Gamma)$. We follow the methods of both sets of authors, utilizing in particular the generating set $\mathcal{P}^0$ of $\Out^0(W_\Gamma)$ defined in \cite{Gutierrez_2012}, and following the structure used in \cite{sale2017outerautomorphismgroupsrightangled} to prove that our set of conditions is necessary and sufficient.

The broader motivation for studying this particular question is to describe the  groups $\Aut(W_\Gamma)$ and $\Out(W_\Gamma)$. These groups are similar to, yet simpler than, $\Aut(F_n)$ and $\Out(F_n)$ (where $F_n$ is a free group with $n$ generators). While there is an immense amount of literature on $\Aut(W_\Gamma)$ and $\Out(W_\Gamma)$, we do not yet have a complete understanding of their structure for all $W_\Gamma$; in particular, little can be said about their geometry. \cite{charney2009automorphismgroupgraphproduct} studied the geometry of $\Aut(W_\Gamma)$ when $\Out(W_\Gamma)$ is finite. The next simplest case is when $\Out(W_\Gamma)$ is a ``small" infinite group, i.e. virtually cyclic, which we discuss in this paper. 

As we will see in \cref{remarkaut}, when $\Out(W_\Gamma)$ is virtually cyclic, $\Aut(W_\Gamma)$ is virtually a group of the form $W_{\Gamma} \rtimes \Z$. The geometry of $W_\Gamma$ is well understood because there is a natural cubical CAT(0) geometry $X$ on which this group acts geometrically. We also understand the geometry of the \emph{direct} product $W_\Gamma \times\Z$ because we can simply take the metric product $X\times\mathbb R$ as a geometry for the direct product. The difficulty is finding a geometry that reflects the algebra of each semi-direct product of $W_{\Gamma}$ with $\Z$. When $\Out(W_\Gamma)$ is virtually cyclic, then up to finite index, it is generated by a single automorphism $\varphi$; thus we only have to understand the geometry of \emph{one} such semi-direct product to uncover the geometry of $\Aut(W_{\Gamma})$. We will not explore geometrical implications of our results, but in general, it is a difficult problem to understand the geometry of $G\rtimes\mathbb Z$ even when the geometry of $G$ is nice.

\subsection{Summary}\label{Section1.2}

In Chapter 2, we list basic definitions and results from graph theory and group theory which are used in the rest of the paper. Chapter 3 includes results from Gutierrez, Piggott, and Ruane (\cite{Gutierrez_2012}), in particular the splitting of $\Aut(W_\Gamma)$, the construction of a generating set $\mathcal{P}^0$ for $\Out^0(W_\Gamma)$, and a set of conditions for when $\Out(W_\Gamma)$ is finite. This chapter also introduces the notion of FSILs and STILs in the graph $\Gamma$ and a set of conditions from Susse and Sale (\cite{sale2017outerautomorphismgroupsrightangled}) for when $\Out(W_\Gamma)$ is large. In Chapter 4, we prove that a pair of distinct SILs in $\Gamma$ corresponds to four distinct partial conjugations in $\Out^0(W_\Gamma)$ which together generate a subgroup that is not virtually $\Z$. Chapter 5 considers separately the case of a disconnected graph $\Gamma$, summarizing methods from \cite{sale2017outerautomorphismgroupsrightangled} and proving the main result for the disconnected case. Chapter 6 concludes with a proof of the main result, a condition on $\Gamma$ so that $\Out(W_\Gamma)$ is virtually $\Z$, for the connected case. We also see a few abstract and concrete examples of what connected graphs satisfy this condition. 

\section{Preliminaries}

Recall some basic definitions about graphs. 

\begin{definition}
    A \textbf{graph} $\Gamma = (V(\Gamma), E(\Gamma))$ is an ordered pair consisting of a set $V(\Gamma)$ of vertices and a set $E(\Gamma)$ of undirected edges which each connect two vertices.

    $\Gamma$ is \textbf{finite} if it has finitely many vertices.

    A \textbf{loop} in $\Gamma$ is an edge connecting a vertex $v \in V(\Gamma)$ to itself.

    \textbf{Multiple edges} are two or more edges connecting the same pair of vertices.

    A \textbf{simple} graph is a graph with no loops or multiple edges.

    An \textbf{induced subgraph} $\Omega$ of $\Gamma$ has as its vertex set a subset $V(\Omega) \subseteq V(\Gamma)$ and as its edge set $E(\Omega) = \{ (v_i,v_j) \in E(\Gamma): v_i, v_j \in V(\Omega) \}$. We may refer to this as just a subgraph, but we always assume a subgraph is induced unless otherwise noted.

    $\Gamma$ is \textbf{complete} if every pair of vertices in $V(\Gamma)$ is joined by an edge.

    A subgraph $C \subseteq \Gamma$ is a \textbf{connected component} of $\Gamma$ if there is a path in $C$ between any two vertices of $C$, and if for any vertex $v \in V(\Gamma) \backslash V(C)$, there is no path in the subgraph $C \cup \{v\}$ between $v$ and a vertex in $C$.

    $\Gamma$ is \textbf{disconnected} if it has more than one connected component. 

    The \textbf{center} of $\Gamma$, denoted $Z(\Gamma)$, is the subset of $V(\Gamma)$ consisting of those vertices which are adjacent to all other vertices in $V(\Gamma)$.
    
    \end{definition}

In this paper, we will assume graphs are simple, finite, and non-trivial. As in the paper by Gutierrez, Piggott, and Ruane (\cite{Gutierrez_2012}), we will define two categories of graph products. 

\begin{definition}
    Let $\Gamma$ be a graph with vertex set $V(\Gamma)= \{v_1,...,v_n\}$. An \textbf{order map} on $\Gamma$ is a function \[\textbf{m}: \{1,2,..., n\} \rightarrow \{p^\alpha: p \text{ prime}, \alpha \in \N\} \cup \{\infty\}.\]

    The pair $(\Gamma, \textbf{m})$ forms a \textbf{labelled graph}. We define the \textbf{graph product of directly indecomposable cyclic groups} $W(\Gamma, \textbf{m})$, usually denoted $W_\Gamma$, as the group with the following presentation:
    \[W_\Gamma := \langle V(\Gamma)\ | v_i^{\textbf{m(i)}}, [v_j,v_k] \mbox{ for } 1 \le i,j,k \le n \mbox{ and }(v_j,v_k) \in E(\Gamma) \rangle.\]
    By convention, if $\textbf{m}(i) = \infty$, then $v_i$ has infinite order. 
\end{definition}

We know from \cite{Radcliffe_2001} that given a group with a graph product decomposition into primary cyclic groups, this decomposition is unique; \cite{GUTIERREZ_2008} extended this result to all graph products of directly indecomposable cyclic groups. So we have a one-to-one correspondence between labelled graphs $(\Gamma, \textbf{m})$ and graph products of directly indecomposable cyclic groups $W_\Gamma$. Given a generator $v$ of $W_\Gamma$, we often do not distinguish between $v$ as a group element and $v$ as a vertex of the graph. \cite{Gutierrez_2012} proved that the graph products of directly indecomposable cyclic groups are exactly the graph products of finitely generated abelian groups, hence some authors refer to these groups with the latter terminology.

Recall that all directly indecomposable cyclic groups have one of three forms: the trivial group $\{\text{id}\}$, $\Z$, or $\Z/p^k\Z$, where $p$ is a prime number and $k$ is a positive integer. Only $\{\text{id}\}$ and $\Z/p^k\Z$ are also primary cyclic groups, and all primary cyclic groups have one of these two forms, so we want to adjust our definition of $W_\Gamma$ so as to not allow vertex groups of $\Z$. 

\begin{definition}
    We say that $W(\Gamma, \textbf{m})$ on $n$ vertices is a \textbf{graph product of primary cyclic groups} if $\textbf{m}(i) < \infty$ for $1 \le i \le n$. 
\end{definition}

From now on, $W_\Gamma$ will refer strictly to a graph product of primary cylic groups. 

\begin{definition}
    A \textbf{right-angled Coxeter group} is a graph product of primary cyclic groups $W(\Gamma, \textbf{m})$ on $n$ vertices with $\textbf{m}(i) = 2$ for $1 \le i \le n$.
\end{definition}

We immediately see that for elements $a$ and $b$ in a right-angled Coxeter group $W_\Gamma$, either $[a,b]= 1$ or the element $ab$ has infinite order. We will give a simple example of a graph product of primary cyclic groups. 

\begin{example}
Consider the graph $\Gamma$ in \cref{pictureGPPCGex}. Let $\textbf{m}$ be the order map on $\Gamma$ given by $\textbf{m}(1)= 2, \ \textbf{m}(2) = 2$, and $\textbf{m}(3) = 3$. In other words, $v_1$ and $v_2$ both correspond to the group $\Z/2\Z$ and $v_3$ corresponds to $\Z/3\Z$. This determines the graph product of primary cyclic groups \[W_\Gamma = \langle v_1,v_2,v_3\ | v_1^2, v_2^2, v_3^3, [v_1,v_2], [v_1,v_3]\rangle.\] An edge between vertices indicates that their corresponding groups commute, so we clearly see that \[W_\Gamma \cong \Z/2\Z \times (\Z/2\Z * \Z/3\Z).\]

    \begin{figure}
    \centering
    \begin{tikzpicture}
    \filldraw (0,0) circle (0.05) node[below]{$v_3$};
    \filldraw (3,0) circle (0.05) node[below]{$v_2$};
    \filldraw (1.5,2) circle (0.05) node[above]{$v_1$};
    \draw (0,0) -- (1.5,2) -- (3,0);
    \end{tikzpicture}
    \caption{Simple example of a graph product of primary cyclic groups}
    \label{pictureGPPCGex}
    \end{figure}
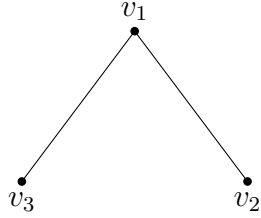

Note that $\Gamma$ as a graph is independent from the order map \textbf{m}; we could define another order map \textbf{m'} on this particular graph which defines a different graph product. However, this graph product will always be of the form $\Z/\textbf{m'}(1)\Z \times (\Z/\textbf{m'}(2)\Z * \Z/\textbf{m'}(3)\Z)$. 

\end{example}

Of course, graph products of primary cyclic groups worth studying are much more complex. A graph theory property that directly impacts the structure of $\Out(W_\Gamma)$ is called a SIL.

\begin{definition}
    In a simple graph $\Gamma$, the \textbf{link} of a vertex $v \in V(\Gamma)$ is the set $Lk(v) := \{w: (v,w) \in E(\Gamma)\}$. The \textbf{star} of $v$ is the set $St(v) := Lk(v) \cup \{v\}$.
\end{definition}

\begin{definition}
    A \textbf{Separating Intersection of Links (SIL)} in a graph is two vertices $v_1$ and $v_2$ in $V(\Gamma)$ satisfying:
    \begin{enumerate}
        \item $d(v_1,v_2) \ge 2$; and
        \item There is a connected component $C \subseteq \Gamma \backslash (Lk(v_1) \cap Lk(v_2))$ such that $v_1$ and $v_2$ are not in $C$. (See \cref{pictureSIL}).
    \end{enumerate}

    There are two equivalent notations for a SIL formed by $v_1$ and $v_2$: we can write either $\{v_1,v_2|C\}$, where $C$ is the connected component satisfying condition (2) of a SIL, or $\{v_1,v_2|z\}$ for some vertex $z$ in $C$. We define two SILs $\{v_1,v_2|C\}$ and $\{w_1,w_2|D\}$ as distinct if $C \cap D = \emptyset$ or if $\{v_1,v_2\} \ne \{w_1,w_2\}$. In the second notation, two SILs $\{v_1,v_2|z\}$ and $\{v_1,v_2|y\}$ are distinct if $z$ and $y$ are not in the same connected component of $\Gamma \backslash (Lk(v_1) \cap Lk(v_2))$, or again if $\{v_1,v_2\} \ne \{w_1,w_2\}$.

    We call the SIL $\{v_1,v_2|C\}$ a \textbf{Coxeter SIL} if $\textbf{m}(1) = \textbf{m}(2) = 2$, or a \textbf{non-Coxeter SIL} if either $\textbf{m}(1)$ or $\textbf{m}(2)$ is greater than 2. We use ``SIL" as an umbrella term to encompass both types. 
\end{definition}

\begin{figure}
    \centering
    \begin{tikzpicture}
    \filldraw (2,4) circle (0.05) node[above]{$v_1$};
    \filldraw (2,1) circle (0.05) node[below]{$v_2$};
    \draw (2,2.5) ellipse (1 and 0.5);
    \node[left] at (1.1,2.5) {$Lk(v_1) \cap Lk(v_2)$};
    \filldraw (1.5,2.5) circle (0.05);
    \filldraw (2,2.5) circle (0.05);
    \filldraw (2.5,2.5) circle (0.05);
    \draw (2,4) -- (1.5,2.5) -- (2,1);
    \draw (2,4) -- (2,2.5) -- (2,1);
    \draw (2,4) -- (2.5,2.5) -- (2,1);
    \draw (4.5,2.5) ellipse (0.5 and 1.5) node{$C$};
    \draw (2.5,2.5) -- (4,2.5);
    \end{tikzpicture}
    \caption{SIL (Separating Intersection of Links)}
    \label{pictureSIL}
\end{figure}
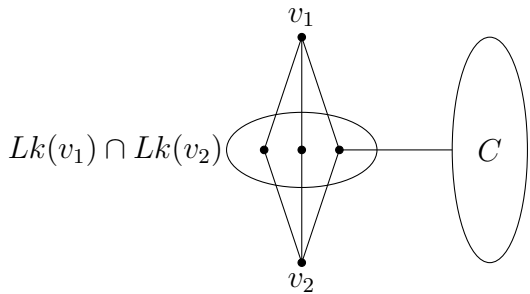

\begin{definition}
    Let $G$ be a group and $P$ a group property. $G$ is \textbf{virtually} $P$ if $G$ contains a finite index subgroup $H$ with the property $P$.
\end{definition}

We will consider when a subgroup of $\Aut(W_\Gamma)$ has one of four algebraic properties: finite, virtually $\Z$ (also called virtually cyclic), virtually abelian, and large. Although the latter three are not  cardinalities of the group, we can roughly think of them as descriptions of the ``size" of the group. Note that a virtually cyclic group is virtually abelian, but otherwise a group cannot have two or more of these properties. \cref{theoremlarge} will establish the strict dichotomy between virtually abelian and large. 

\begin{definition}
    A group $G$ is \textbf{large} if there is a finite index subgroup $H$ of $G$ and a surjective homomorphism $H \rightarrow F_2$, the free group of rank 2.
\end{definition}

\begin{example}
    Let $W_\Gamma = W(\Gamma, \textbf{m})$ be a graph product of primary cyclic groups. Suppose that $v_1$ and $v_2$ are non-adjacent vertices and consider the empty subgraph $\Omega$ induced by $v_1$ and $v_2$. Now, $W_\Omega$ has the presentation $\langle v_1,v_2| v_1^{\textbf{m}(1)}, v_2^{\textbf{m}(2)} \rangle$. Since there is no edge between the two vertices, we get the free product \[W_\Omega \cong (\Z/\textbf{m}(1)\Z) * (\Z/\textbf{m}(2)\Z).\] If $\textbf{m}(1) = \textbf{m}(2) = 2$, then  \[W_\Omega \cong (\Z/2\Z) * (\Z/2\Z) \cong D_\infty.\] If either $\textbf{m}(1)$ or $\textbf{m}(2)$ is greater than 2, say $\textbf{m}(1) >2$, then $W_\Omega$ contains the subgroup $\langle v_1v_2, v_1^2v_2 \rangle$, which is isomorphic to $F_2$. In that case, $W_\Omega$ is large.

    Now suppose there is another vertex $v_3$ of $\Gamma$ and that the subgraph $\Delta \subseteq \Gamma$ induced by $v_1, v_2$, and $v_3$ is empty. Then \[W_\Delta \cong (\Z/\textbf{m}(1)\Z) * (\Z/\textbf{m}(2)\Z) * (\Z/\textbf{m}(3)\Z).\] For any values of $\textbf{m}(1), \textbf{m}(2)$, and $\textbf{m}(3)$, we have a subgroup $\langle v_1v_2, v_1v_3 \rangle$ in $W_\Delta$. This subgroup is isomorphic to $F_2$, so $W_\Delta$ is large. 

    Note that if a general group $G$ is large, then $\Inn(G)$ is large, and thus $\Aut(G)$ is large. 
    
\end{example}

We conclude this chapter with a few facts that we will be referring to as tools. We won't prove them here, but leave them as exercises to the reader. 

\begin{fact}
    $\Z/2\Z * \Z/2\Z \cong D_\infty$.
\end{fact}

\begin{fact}
    $D_\infty$ is virtually $\Z$.
\end{fact}

\begin{fact}
    $D_\infty \times D_\infty$ is virtually $\Z \times \Z$, hence it is not virtually $\Z$.
\end{fact}

\section{Existing results about $\Aut(W_\Gamma)$ and $\Out(W_\Gamma)$}

\subsection{The splitting of $\Aut(W_\Gamma)$}\label{Section3.1}

To motivate our study of the outer automorphism group $\Out(W_\Gamma)$, we will see how the entire automorphism group $\Aut(W_\Gamma)$ splits as a semidirect product.

\begin{definition}
    $\Inn(W_\Gamma)$ is the subgroup of $\Aut(W_\Gamma)$ generated by inner conjugations by vertices in $\Gamma$. An \textbf{inner conjugation} by an element $v \in \Gamma$ is the map $\varphi_v: W_\Gamma \rightarrow W_\Gamma$ given by $\varphi_v(x) = vxv^{-1}$ for all $x \in W_\Gamma$. 
\end{definition}

Note that if $v$ and $w$ are adjacent vertices, meaning they commute as group elements, then the inner conjugation of $w$ by $v$ is $vwv^{-1} = wvv^{-1} = w$. Thus $\varphi_v$ is trivial on $St(v)$. 

\begin{definition}
    Let $v$ be a vertex of $\Gamma$ and let $C$ be a connected component of the subgraph $\Gamma \backslash St(v)$. The \textbf{partial conjugation} of $C$ by $v$ is the group action denoted $\chi_{v, C}$ and is defined:
    \[\chi_{v,C}(w) = 
    \begin{cases} 
     vwv^{-1} & \mbox{if } w \in C\\
     w & \mbox{if } w \notin C
     \end{cases}\]
      
    $\chi_{v,C}$ is a \textbf{non-trivial partial conjugation} if $C \subsetneqq \Gamma \backslash St(v)$, and it is a \textbf{trivial partial conjugation} if $C = \Gamma \backslash St(v)$, that is, $\chi_{v,C}$ is an inner conjugation of $\Gamma$ by $v$. We denote the set of all partial conjugations of $\Gamma$ by $\mathcal{P}$. 
\end{definition}

Let $C_1, ..., C_n$ be all disjoint connected components in $\Gamma \backslash St(v)$. Then the product of partial conjugations $\chi_{v,C_1} \circ \chi_{v,C_2} \circ ... \circ \chi_{v,C_n}$ is an inner conjugation of $\Gamma$ by $v$, thus a trivial partial conjugation. 

\begin{definition}
We now define relevant subgroups of $\Aut(W_\Gamma)$ as in \cite{Gutierrez_2012}.                          
\begin{itemize}
    \item The group $\Aut^0(W_\Gamma)$ is those automorphisms of $W_\Gamma$ which map each vertex $v \in V(\Gamma)$ to a conjugate of itself. 
    \item $\Aut^1(W_\Gamma)$ is those automorphisms of $W_\Gamma$ which map each maximal complete subgroup of $W_\Gamma$ to a maximal complete subgroup; in other words, they map maximal cliques of $\Gamma$ to maximal cliques. 
    \item $\Out(W_\Gamma)$ is the quotient $\Aut(W_\Gamma) / \Inn(W_\Gamma)$. 
    \item $\mathcal{O}ut^0(W_\Gamma)$ is the quotient $\Aut^0(W_\Gamma) / \Inn(W_\Gamma)$.
    \item $\Out^0(W_\Gamma)$ is the group generated by the set $\mathcal{P}^0 \subseteq \mathcal{P}$ of partial conjugations of $\Gamma$. 
\end{itemize}
\end{definition}

We will define the set $\mathcal{P}^0$ later on. 

\begin{lemma}[Corollary 4.5 \cite{Gutierrez_2012}] \label{cor4.5}
    The set $\mathcal{P}$ of partial conjugations of $\Gamma$ is a minimal generating set for $\Aut^0(W_\Gamma)$.
\end{lemma}

Combining results from \cite{Gutierrez_2012} and Laurence (\cite{Laurence_1995}), we obtain:
\begin{theorem} \label{theoremsplitting}
    If $W_\Gamma$ is a graph product of primary cyclic groups, then \[\Aut(W_\Gamma) = \Aut^0(W_\Gamma) \rtimes \Aut^1(W_\Gamma) = (\Inn(W_\Gamma) \rtimes \Out^0(W_\Gamma)) \rtimes \Aut^1(W_\Gamma)\] and $\Aut^1(W_\Gamma)$ is a finite group. 
\end{theorem}

\begin{remark}\label{remarkaut}
    What can we say about $\Inn(W_\Gamma)$ in this splitting? As remarked by the authors of \cite{charney2009automorphismgroupgraphproduct}, $\Inn(W_\Gamma)$ turns out to share roughly the structure of $W_\Gamma$. Recall that only vertices in $\Gamma \backslash Z(\Gamma)$ act in a non-trivial inner automorphism of $W_\Gamma$. Let $\Omega$ be the subgraph induced by $\Gamma \backslash Z(\Gamma)$, then $W_\Omega \cong \Inn(W_\Gamma)$ and $W_\Gamma = W_\Omega \times W_{Z(\Gamma)}$. $W_{Z(\Gamma)}$ is a direct product of finitely many finite groups, so $W_\Omega$ is a finite index subgroup of $W_\Gamma$. Therefore $\Inn(W_\Gamma)$ is isomorphic to a finite index subgroup of $W_\Gamma$ and $\Aut(W_\Gamma)$ is, up to finite index, a group of the form $W_\Gamma \rtimes \Out^0(W_\Gamma)$. 

\end{remark}

To understand $\mathcal{O}ut^0(W_\Gamma)$ as a subgroup of $\Out(W_\Gamma)$, we have the following lemma. 

\begin{lemma} \label{lemmaoutsplit}
    Let $\Gamma$ be a graph product of directly indecomposable groups. Then $\mathcal{O}ut^0(W_\Gamma)$ is isomorphic to a normal subgroup of $\Out(W_\Gamma)$ and \[\Out(W_\Gamma)/\ \mathcal{O}ut^0(W_\Gamma) \cong \Aut^1(W_\Gamma).\]
\end{lemma}

\begin{proof}
    By definition, $\Out(W_\Gamma) = \Aut(W_\Gamma) / \Inn(W_\Gamma)$ and \newline $\mathcal{O}ut^0(W_\Gamma) = \Aut^0(W_\Gamma) / \Inn(W_\Gamma)$. We know that $\Inn(W_\Gamma)$ is a normal subgroup of $\Aut^0(W_\Gamma)$ and $\Aut^0(W_\Gamma)$ is a normal subgroup of $\Aut(W_\Gamma)$. By the third isomorphism theorem, the quotient group $\Aut(W_\Gamma)/\Inn(W_\Gamma) = \Out(W_\Gamma)$ has a normal subgroup isomorphic to the quotient group \newline $\Aut^0(W_\Gamma)/\Inn(W_\Gamma) = \mathcal{O}ut^0(W_\Gamma)$. The third isomorphism theorem also implies that \[(\Aut(W_\Gamma) / \Inn(W_\Gamma)) / (\Aut^0(W_\Gamma) / \Inn(W_\Gamma)) \cong \Aut(W_\Gamma) / \Aut^0(W_\Gamma)\] \[\Rightarrow \Out(W_\Gamma) / \mathcal{O}ut^0(W_\Gamma) \cong \Aut^1(W_\Gamma).\] 
\end{proof}

Only when $W_\Gamma$ is the graph product of primary cyclic groups do we get that $\Aut^0(W_\Gamma)$ is a finite index subgroup of $\Aut(W_\Gamma)$, which reduces all automorphisms of $W_\Gamma$ to conjugations of $\Gamma$, up to finite index. We emphasize that this is not the case for graph products of directly indecomposable cyclic groups in general. Moreover, the last statement of \cref{theoremfinite} (below) fails if $W_\Gamma$ is not a graph product of primary cyclic groups, which is why our results hold only for graph products of primary cyclic groups.

By \cref{lemmaoutsplit}, if $W_\Gamma$ is a graph product of primary cyclic groups, then $\mathcal{O}ut^0(W_\Gamma)$ is a finite index subgroup of $\Out(W_\Gamma)$; if $\mathcal{O}ut^0(W_\Gamma)$ is virtually $\Z$, then $\Out(W_\Gamma)$ is virtually $\Z$. And if $\Out(W_\Gamma)$ is virtually $\Z$, then each of its infinite subgroups, including $\Out^0(W_\Gamma)$, must be virtually $\Z$. Similarly, $\Out(W_\Gamma)$ is virtually abelian (or large) if and only if $\Out^0(W_\Gamma)$ is virtually abelian (respectively, large). For this reason, some statements in this paper may discuss $\Out(W_\Gamma)$ being virtually $\Z$, and others use that $\mathcal{O}ut^0(W_\Gamma)$ is virtually $\Z$, but all results tell us about the structure of $\Out(W_\Gamma)$. This fact and the usefulness of the set $\mathcal{P}^0$ motivates our study of the group $\mathcal{O}ut^0(W_\Gamma)$ in particular.

\begin{definition}
    A vertex $b \in V(\Gamma)$ is a \textbf{star cut point} if $\Gamma \backslash St(b)$ contains at least two distinct connected components. (See \cref{picturecutpoint}).
\end{definition}

Note that if $\Gamma$ contains the SIL $\{v_1,v_2|C\}$, then both $v_1$ and $v_2$ are star star cut points. The set of star cut points is exactly the set of vertices which act in a non-trivial partial conjugation. 

\begin{figure}
\centering
\begin{tikzpicture}
    \filldraw (2.5,2.5) circle (0.05) node[above]{$b$};
    \filldraw (4,3) circle (0.05);
    \filldraw (4,2) circle (0.05);
    \draw (4,3) -- (2.5,2.5) -- (4,2);
    \draw (4.5,2.5) ellipse (0.5 and 1.5) node{$C_2$};
    \filldraw (1,2) circle (0.05);
    \filldraw (1,3) circle (0.05);
    \draw (1,2) -- (2.5,2.5) -- (1,3);
    \draw (0.5,2.5) ellipse (0.5 and 1.5) node{$C_1$}; 
\end{tikzpicture}
\caption{Star cut point}    
\label{picturecutpoint}
\end{figure}
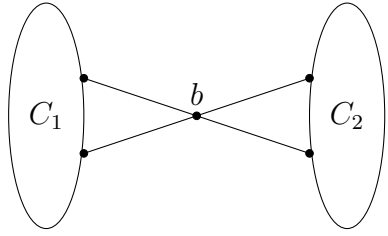

\begin{definition}
    \cite{Gutierrez_2012} defines the set $\mathcal{P}^0$ for $W_\Gamma$ as follows. Number all the vertices of $\Gamma$ in any manner. Note that each star cut point in $\Gamma$ is the acting vertex of at least two non-trivial partial conjugations of $W_\Gamma$. Take a star cut point $v \in V(\Gamma)$ and index the $n$ connected components of $\Gamma \backslash St(v)$ such that the component with the smallest numbered vertex is indexed $C_1$, the remaining component with the next smallest vertex is $C_2$, and so on (see an example in  \cref{pictureconstructP0}). Then we have the partial conjugations $\chi_{v, C_1}, \chi_{v, C_2}, ... ,\chi_{v, C_n}$ in $\Out^0(W_\Gamma)$. Now put all $\chi_{v, C_i}$ with $i \ge 2$ in the set $\mathcal{P}^0$. We get $n-1$ partial conjugations with acting vertex $v$ in $\mathcal{P}^0$. Repeat for each star cut point in $V(\Gamma)$. Recall that we define $\Out^0(W_\Gamma)$ as the set generated by $\mathcal{P}^0$. Intuitively, this construction stops short of generating an inner conjugation by any star cut point. If we included all the partial conjugations $\chi_{v,C_1}, ..., \chi_{v,C_n}$ in $\mathcal{P}^0$, then their product would be an inner conjugation, which is trivial as a partial conjugation. 
\end{definition}

Theorem 1.7 of \cite{Gutierrez_2012} shows that the isomorphism $\Out^0(W_\Gamma) \cong \mathcal{O}ut^0(W_\Gamma)$ holds regardless of how the vertices are numbered. For simplicity, we will treat them as the same group $\Out^0(W_\Gamma)$, so that we can consider the group generated by $\mathcal{P}^0$ as a finite index subgroup of $\Out(W_\Gamma)$, but we emphasize that they are different groups. Thus if we can find an appropriate numbering of the vertices of $\Gamma$, we can make certain partial conjugations be in $\mathcal{P}^0$ and prove structural properties of $\Out^0(W_\Gamma)$ and hence of $\Out(W_\Gamma)$. See a concrete example of the construction of $\mathcal{P}^0$ in \cref{examplemainthm}.

\begin{figure}
    \centering
    \begin{tikzpicture}
        \filldraw (5,5) circle (0.05) node[below]{$v_5$};
        \filldraw (7,6) circle (0.05) node[right]{$v_1$};
        \filldraw (3,6) circle (0.05) node[above]{$v_8$};
        \filldraw (3,4) circle (0.05) node[below]{$v_{10}$};
        \filldraw (7,4) circle (0.05) node[above]{$v_7$};
        \draw (7,6) -- (5,5) -- (7,4);
        \draw (3,6) -- (5,5) -- (3,4);
        \draw (8.3,7.5) circle (1);
        \node[above] at (9,8.2) {$C_3$};
        \filldraw (8,8) circle (0.05) node[above]{$v_{12}$};
        \draw (7,6) -- (8,8) -- (9,7);
        \filldraw (9,7) circle (0.05) node[above]{$v_6$};
        \draw (1.5,7.3) circle (1);
        \node[left] at (1,8.5) {$C_2$};
        \filldraw (1,6.5) circle (0.05) node[above]{$v_3$};
        \filldraw (2,8) circle (0.05) node[left]{$v_9$};
        \draw (3,6) -- (1,6.5) -- (2,8) -- (3,6);
        \draw (1.7,2.5) circle (1);
        \node at (1,3.6) {$C_1$};
        \filldraw (2,3) circle (0.05) node[above]{$v_2$};
        \filldraw (1,2) circle (0.05) node[above]{$v_4$};
        \draw (3,4) -- (2,3) -- (1,2);
        \draw (8.5,2.5) circle (1);
        \node at (9,3.6) {$C_4$};
        \filldraw (8.5,2.5) circle (0.05) node[above]{$v_{11}$};
        \draw (7,4) -- (8.5,2.5);
    \end{tikzpicture}
    \caption[An example of the construction of $\mathcal{P}^0$]{An example of the construction of $\mathcal{P}^0$. We consider some indexing of $V(\Gamma)$ and a star cut point $v_5 \in V(\Gamma)$.}
    \label{pictureconstructP0}
\end{figure}
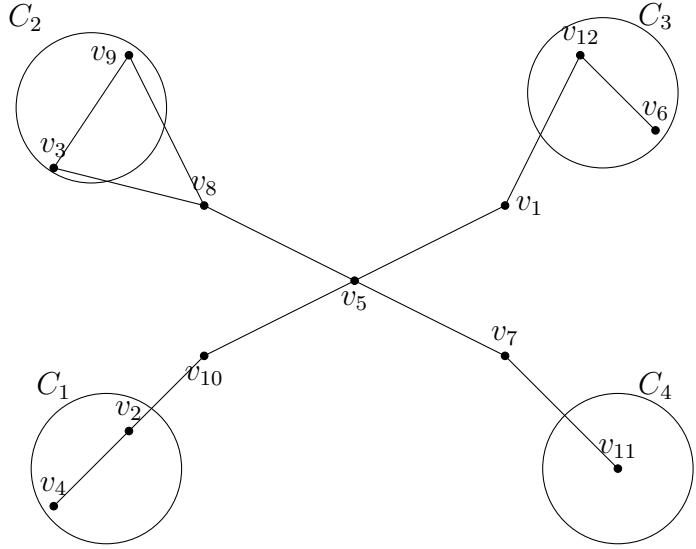

\subsection{Determining when $\Out(W_\Gamma)$ is finite or large}\label{Section3.2}

First, we concretely connect the presence of a SIL with a pair of partial conjugations of $\Gamma$. The statement and proof of this lemma are by Hikaru Jitsukawa.

\begin{lemma}\label{lemma2.2}
If $\Gamma$ contains a SIL $\{v_1,v_2|z\}$, then there is a connected component $C \subseteq \Gamma$ such that $z$ is in $C$, $v_1$ and $v_2$ are not in $C$, and both $\chi_{v_1, C}$ and $\chi_{v_2, C}$ are partial conjugations in $\mathcal{P}$.
\end{lemma}

\begin{proof}

Suppose there is a vertex $z \in V(\Gamma)$, a connected component $C_1 \subseteq \Gamma \backslash Lk(v_1)$, and a connected component $C_2 \subseteq \Gamma \backslash Lk(v_2)$ such that $z \in C_1 \cap C_2$. For the sake of contradiction, assume $C_1 \ne C_2$, and without loss of generality, assume there is $y \in C_1 \backslash C_2$. Since $y \in C_1$, there is a path $\gamma$ in $\Gamma$ between $y$ and $z$ such that no point passes through $St(v_1)$. But since $y \notin C_2$, $\gamma$ must intersect $St(v_2)$ at some point $a$. Then there is a path $\gamma'$ between $a$ and $z$ which does not intersect $St(v_1)$. But this means that $v_2$ and $z$ are in the same path component of $\Gamma \backslash (Lk(v_1) \cap Lk(v_2))$, contradicting our assumption of a SIL. Therefore we must have that $C_1 = C_2$, and $\chi_{v_1,C}$ and $\chi_{v_2,C}$ are well-defined. (See \cref{pictureforlem2.2}).

\begin{figure}
    \centering
    \begin{tikzpicture}
        \filldraw (2,8) circle (0.05) node[above]{$v_1$};
        \filldraw (2,2) circle (0.05) node[below]{$v_2$};
        \filldraw (1,5) circle (0.05);
        \filldraw (2,5) circle (0.05);
        \filldraw (3,5) circle (0.05);
        \draw (2,8) -- (1,5) -- (2,2);
        \draw (2,8) -- (2,5) -- (2,2);
        \draw (2,8) -- (3,5) -- (2,2);
        \draw (2,5) ellipse (1.75 and 0.5);
        \node at (0.5,5.75) {$Lk(v_1) \cap Lk(v_2)$};
        \draw (6.5,6) ellipse (2 and 2.5);
        \node[above] at (8,8) {$C_1$};
        \filldraw (4,8) circle (0.05);
        \filldraw (5.5,7.5) circle (0.05);
        \draw (2,8) -- (4,8) -- (5.5,7.5);
        \draw (6,4.5) ellipse (2 and 2.5);
        \filldraw (5,5.5) circle (0.05);
        \draw (3,5) -- (5,5.5);
        \node[below] at (7,2) {$C_2$};
        \filldraw (4,2) circle (0.05);
        \filldraw (5,3) circle (0.05);
        \draw (2,2) -- (4,2) -- (5,3);
        \filldraw (6,4) circle (0.05) node[below]{$z$};
        \filldraw (6.5,6) circle (0.05) node[left]{$a$};
        \filldraw (7,8) circle (0.05) node[left]{$y$};
        \draw (7,8) -- (6.5,6) -- (6,4);
        \node at (6,5) {$\gamma'$};
        \draw[red][thick] (2,2) -- (6.5,6);
    \end{tikzpicture}
    \caption[Picture for \cref{lemma2.2}]{Picture for \cref{lemma2.2}. We prove that the red edge connecting $v_2$ and $a$ must exist.}
    \label{pictureforlem2.2}
\end{figure}

\end{proof}

This confirms that each SIL in $\Gamma$ corresponds to a pair of non-trivial partial conjugations in $\mathcal{P}$. A useful result from (\cite{Gutierrez_2012}) connects the presence of a SIL with the cardinality of $\Out(W_\Gamma)$.

\newpage

\begin{theorem} [Corollary 1.11 \cite{Gutierrez_2012}] \label{theoremfinite}
    If $W_\Gamma$ is a graph product of primary cyclic groups, then the following are equivalent:
    \begin{enumerate}
        \item $\Out^0(W_\Gamma)$ is an abelian group.
        \item $\Gamma$ does not contain a SIL.
        \item $\Out(W_\Gamma)$ is finite. 
    \end{enumerate}
\end{theorem}

Therefore, in our search for graphs $\Gamma$ with $\Out(W_\Gamma)$ virtually $\Z$, we know that we need $\Gamma$ to contain at least one SIL. A logical next question is: how many Coxeter SILs can $\Gamma$ contain before $\Out^0(W_\Gamma)$ is virtually abelian and not virtually $\Z$? We will address the number of SILs more in Chapter 4. Now, there are two related graph properties that Susse and Sale (\cite{sale2017outerautomorphismgroupsrightangled}) have proven to make $\Out(W_\Gamma)$ large.

\begin{definition}
    Three vertices $v_1, v_2, v_3 \in V(\Gamma)$ form a \textbf{Separating Triple Intersection of Links (STIL)} if
    \begin{itemize}
        \item The subgraph spanned by $\{v_1,v_2,v_3\}$ contains at most one edge; and
        \item There is a connected component $C \subseteq \Gamma \backslash (L_{v_1} \cap L_{v_2} \cap L_{v_3})$ such that $v_1, v_2,$ and $v_3$ are not in $C$.
    \end{itemize}

See \cref{pictureSTIL}.
\end{definition}

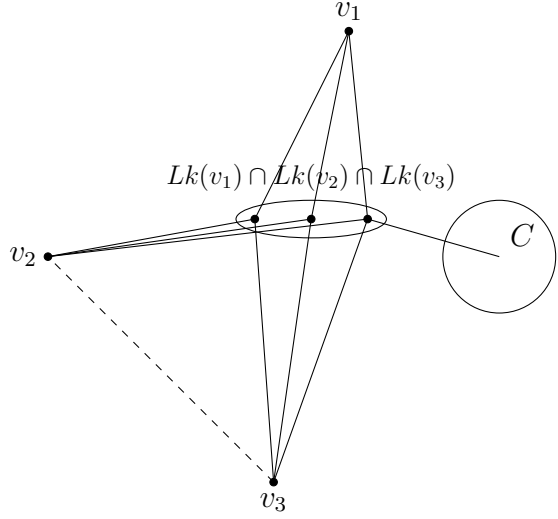
\begin{figure}
    \centering
    \begin{tikzpicture}
    \draw (3.5,3.5) ellipse (1 and 0.25) node[scale=0.85][above=0.3 cm]{$Lk(v_1) \cap Lk(v_2) \cap Lk(v_3)$}; 
    \filldraw (2.75,3.5) circle (0.05);
    \filldraw (3.5,3.5) circle (0.05);
    \filldraw (4.25,3.5) circle (0.05);
    \draw (6,3) circle (0.75) node[above right]{$C$};
    \draw (4.25,3.5) -- (6,3);
    \filldraw (4,6) circle (0.05) node[above]{$v_1$};
    \draw (4,6) -- (2.75,3.5);
    \draw (4,6) -- (3.5,3.5);
    \draw (4,6) -- (4.25,3.5);
    \filldraw (0,3) circle (0.05) node[left]{$v_2$};
    \draw (0,3) -- (2.75,3.5);
    \draw (0,3) -- (3.5,3.5);
    \draw (0,3) -- (4.25,3.5);
    \filldraw (3,0) circle (0.05) node[below]{$v_3$};
    \draw (3,0) -- (2.75,3.5);
    \draw (3,0) -- (3.5,3.5);
    \draw (3,0) -- (4.25,3.5);
    \draw[dashed] (3,0) -- (0,3);
    \end{tikzpicture}
    \caption{STIL (Separating Triple Intersection of Links)}
    \label{pictureSTIL}
\end{figure}

\begin{definition}
    Three vertices $v_1, v_2, v_3 \in V(\Gamma)$ form a \textbf{Flexible Separating Intersection of Links (FSIL)} if $\{v_i,v_j\ |\ v_k\}$ forms a SIL for $\{i,j,k\} = \{1,2,3\}$. 

\end{definition}

\begin{theorem}[Theorem 2 \cite{sale2017outerautomorphismgroupsrightangled}] \label{theoremlarge} 
    Let $W_\Gamma$ be a graph product of primary cyclic groups. Then the following are equivalent:
    \begin{enumerate}
        \item $\Out(W_\Gamma)$ is large;
        \item $\Out(W_\Gamma)$ is not virtually abelian;
        \item $\Gamma$ contains a STIL, an FSIL, or a non-Coxeter SIL. 
    \end{enumerate}
    In particular, $\Out(W_\Gamma)$ obeys a strict dichotomy: it is either virtually abelian or large. 
\end{theorem}

\section{Connected graphs with two distinct Coxeter SILs}

In this chapter, we look at the case where $\Gamma$ is a connected graph containing two distinct Coxeter SILs. As we will see, the partial conjugations resulting from these SILs generate an infinite subgroup of $\Out^0(W_\Gamma)$ which is not virtually $\Z$. A connected graph $\Gamma$ can have two SILs $\{v_1,v_2|C\}$ and $\{w_1,w_2|D\}$ with overlapping acting vertices (say, $v_1 = w_1$ and/or $v_2 = w_2$) and still have $\Out(W_\Gamma)$ be virtually abelian; this is not true in a disconnected graph. For this reason, it is practical to split some of our proofs into connected and disconnected.  

The layout of this chapter and the outline of the proof of \cref{theorem3.5} are written by Hikaru Jitsukawa at Tufts. I edited these components and added details where necessary.

\cref{lemma2.2} allowed us to associate a SIL in $\Gamma$ with a pair of partial conjugations in $\Out^0(W_\Gamma)$. The following lemma extends that statement to two or more SILs. 

\begin{lemma}\label{prop}
    Let $W_\Gamma$ be a graph product of primary cyclic groups, where $\Gamma$ is a connected graph containing two distinct Coxeter SILs $\{v, v'|C_v\}$ and $\{w,w'|C_w\}$. Denote the corresponding partial conjugations $\chi_{v,C_v}, \chi_{v',C_v}, \chi_{w,C_w}$, and $\chi_{w',C_w}$, as in \cref{lemma2.2}. Then all four of these partial conjugations are in $\Out^0(W_\Gamma)$. 
    
\end{lemma}

\begin{proof}
    Assume $\mathcal{P}^0$ has been constructed with an arbitrary ordering of the vertices of $\Gamma$. By definition, $\Out^0(W_\Gamma) = \langle \mathcal{P}^0 \rangle$. Suppose that $\chi_{v,C_v}$ is not in $\mathcal{P}^0$. Since $v$ is a star cut point, there is at least one partial conjugation in $\mathcal{P}^0$ with acting vertex $v$. In fact, every other partial conjugation with acting vertex $v$ lies in $\mathcal{P}^0$, by the construction of $\mathcal{P}^0$; denote these $\chi_{v,C_2}, ..., \chi_{v,C_n}$. Inner conjugations are trivial in $\Out(W_\Gamma)$, so the composition of automorphisms $\chi_{v,C_v} \circ \chi_{v,C_2} \circ ... \circ \chi_{v,C_n} = 1$ in $\Out(W_\Gamma)$. Then $\chi_{v,C_v} = (\chi_{v,C_2} \circ ... \circ \chi_{v,C_n})^{-1}$ in $\Out(W_\Gamma)$. $(\chi_{v,C_2} \circ ... \circ \chi_{v,C_n})^{-1}$ is in $\Out^0(W_\Gamma)$,  therefore $\chi_{v,C_v}$ is in $\Out^0(W_\Gamma)$. 

    The same argument shows that $\chi_{v',C_v}, \chi_{w,C_w}$, and $\chi_{w',C_w}$ are in $\Out^0(W_\Gamma)$.
\end{proof}

We will see that these four partial conjugations make $\Out^0(W_\Gamma)$ ``too big" to be virtually $\Z$. First, we state a few more definitions and auxiliary lemmas.

\begin{lemma}[Lemma 1.7 \cite{sale2017outerautomorphismgroupsrightangled}] \label{lemma1.7}
    Suppose that $\Gamma$ is connected and contains the SILs $\{x_1,x_2|Z\}$ and $\{x_1,x_3|Z'\}$. If $z \in Z \cap Z'$, then $\{x_1,x_2,x_3|z\}$ is a STIL. 
\end{lemma}

\begin{definition}
Let $S$ be a set. A \textbf{word} in $S$ is a product $x_{i_1}^{a_1} x_{i_2}^{a_2}...x_{i_n}^{a_n}$ for some elements $x_{i_1}, ..., x_{i_n}$ of $S$, where $x_{i_j} \ne x_{i_{j+1}}$, and some non-zero integers $a_1,...,a_n$. We call $S$ an \textbf{alphabet}.
\end{definition}

\begin{lemma}[Lemma 2.4 \cite{Gutierrez_2012}] \label{subgraph}
    Let $\Omega$ be a subgraph of $\Gamma$. The natural map $W_\Omega \rightarrow W_\Gamma$ is an embedding.
\end{lemma}

Thus, taking any subgraph $\Omega \subseteq \Gamma$, we get a subgroup $W_\Omega$ of $W_\Gamma$, which is itself a graph product of primary cyclic groups. 

\begin{lemma}[Corollary 1.6 \cite{Gutierrez_2012}] \label{alphabet}

Let $\Gamma$ be a graph and let $\Omega \subseteq \Gamma$ be an induced subgraph. Let $\mathcal{P}_\Omega$ be the set of partial conjugations of $\Gamma$ by a vertex in $\Omega$. If $\varphi \in \Aut^0(W_\Gamma)$ and if there exist words $z_1, ..., z_n \in W_\Omega$ such that $\varphi(v_j) = z_jv_jz_j^{-1}$ for each $1 \le j \le n$, then any word for $\varphi$ in the alphabet $\mathcal{P}^{\pm1}$ can be rewritten as a word in the alphabet $\mathcal{P}_\Omega^{\pm1}$ (a word which still spells $\varphi$) by simply omitting those generators not in $\mathcal{P}_\Omega^{\pm1}$.
    
\end{lemma}

\begin{proposition}\label{theorem3.5}
    Let $W_\Gamma$ be a graph product of primary cyclic groups. If $\Gamma$ is connected and contains two or more distinct Coxeter SILs, then $\Out^0(W_\Gamma)$ is not virtually $\Z$.
\end{proposition}

\begin{proof}

This is an edited version of Hikaru Jitsukawa's proof. If $\Out(W_\Gamma)$ is not virtually abelian, then clearly $\Out^0(W_\Gamma)$ is not virtually $\Z$. So we assume $\Out(W_\Gamma)$ is virtually abelian. We take any two distinct SILs in $\Gamma$ and denote them $\{v_1,v_2|C\}$ and $\{w_1,w_2|D\}$. By \cref{prop}, $\chi_{v_1,C}, \chi_{v_2,C}, \chi_{w_1,D}$, and $\chi_{w_2,D}$ are in $\Out^0(W_\Gamma)$. Define the subgroups $A = \langle \chi_{v_1,C}, \chi_{v_2,C} \rangle$ and $B = \langle \chi_{w_1,D}, \chi_{w_2,D} \rangle$ of $\Out^0(W_\Gamma)$. Note that $A \cong \Z/2\Z * \Z/2\Z \cong D_\infty$ and \newline $B \cong \Z/2\Z * \Z/2\Z \cong D_\infty$. We will show that $A$ and $B$ have trivial intersection.

Suppose, for the sake of contradiction, that there is a non-trivial element $\varphi \in A \cap B$. First, we consider the case where $\{v_1,v_2\} \cap \{w_1,w_2\} = \emptyset$. Since $\varphi \in A$, there are elements $\chi_1,...,\chi_n \in A$, with each $\chi_i$ equal to either $\chi_{v_1,C}$ or $\chi_{v_2,C}$, such that $\varphi = \chi_1 \circ ...\circ \chi_n$. Then for some $v \in V(\Gamma)$, we can write $\varphi(v) = g_ng_{n-1}...g_1(v)g_1...g_{n-1}g_n$, where each $g_i$ is either $v_1$ or $v_2$. 

Since $\varphi \in B$, we can also write $\varphi(v) = h_mh_{m-1}...h_1(v)h_1...h_{m-1}h_m$, where each $h_i$ is either $w_1$ or $w_2$. By \cref{alphabet}, the word $h_m...h_1(v)h_1...h_m$ can be rewritten in the alphabet $\mathcal{P}_{\{v_1,v_2\}}^{\pm1}$, by omitting any generators besides $v_1$ and $v_2$. But then we delete every letter of $h_m...h_1(v)h_1...h_m$ and are left with the empty word. Similarly, we can rewrite the word $g_n...g_1(v)g_1...g_n$ in the alphabet $\mathcal{P}_{\{w_1,w_2\}}^{\pm1}$, and are left with the empty word. Thus $\varphi(v) = v$ for all $v \in V(\Gamma)$, so $\varphi$ is the identity on $\Gamma$. 

Now suppose the two SILs share at least one vertex; without loss of generality, suppose $v_1 = w_1$. By \cref{lemma1.7}, if $C \cap D \ne \emptyset$, then $\Gamma$ contains a STIL, so by \cref{theoremlarge}, $\Out(W_\Gamma)$ is large. So we may assume $C \cap D = \emptyset$. If the two SILs share both acting vertices, so $v_1=w_1$ and $v_2=w_2$, then we also must have $C \cap D = \emptyset$ for the SILs to be distinct. 

For some vertex $v$, if $v$ is in $\Gamma \backslash (C \cup D)$, then $\varphi(v) = v$. If $v \in C$, then since $\varphi \in B$, it is a word in partial conjugations acting on $D$, so $\varphi(v) = v$. On the other hand, if $v \in D$, then since $\varphi \in A$, it is a word in partial conjugations acting on $C$, so $\varphi(v) = v$. Thus $\varphi$ is the identity on $\Gamma$. 

Then in all cases, we have $A \cap B = \{\text{id}\}$. Now suppose for the sake of contradiction that $\Out^0(W_\Gamma)$ has a finite index subgroup $K \cong \Z$. Since $A$ and $B$ are both infinite subgroups in $\Out^0(W_\Gamma)$, then both $A \cap K$ and $B \cap K$ are infinite subgroups of finite index in $\Out^0(W_\Gamma)$. $A \cap K$ and $B \cap K$ must also have non-trivial intersection. On the other hand, $(A \cap K) \cap (B \cap K) = (A \cap B) \cap K = \{\id\}$, which is a contradiction. Therefore $\Out^0(W_\Gamma)$ is not virtually $\Z$. 
\end{proof}

\section{Disconnected graphs}

We study the disconnected case separately. As we will see, 
when $\Gamma$ is disconnected, we have to place stronger restrictions on the structure of $\Gamma$ to keep $\Out^0(W_\Gamma)$ (or equivalently $\Out(W_\Gamma))$ from being large. We can then describe the family of disconnected graphs with exactly one SIL, and prove that in this case, a more precise statement holds than for the connected case: $\Out^0(W_\Gamma) \cong D_\infty$.

\subsection{Summarizing a proof by Susse and Sale}

Since our proof of \cref{lemma5} will be an application of the methods Susse and Sale use to prove their Proposition 2.3 in \cite{sale2017outerautomorphismgroupsrightangled}, it is worthwhile to summarize their proof here.

\begin{proposition}[Proposition 2.3 \cite{sale2017outerautomorphismgroupsrightangled}] \label{prop2.3}
    Suppose $W_\Gamma$ is a right-angled Coxeter group defined by a disconnected graph $\Gamma$ which contains no FSIL or STIL. Then $\Out^0(W_\Gamma)$ is a virtually abelian right-angled Coxeter group. 
\end{proposition}

\begin{proof}

We summarize the proof that Susse and Sale give of this proposition in \cite{sale2017outerautomorphismgroupsrightangled}. Susse and Sale first describe the disconnected graphs $\Gamma$ which contain no FSIL or STIL. They prove that if $\Gamma$ contains three (or more) connected components $\Gamma_1, \Gamma_2,$ and $\Gamma_3$, then we can take $x_i \in \Gamma_i$ and form an FSIL $\{x_1,x_2,x_3\}$. Then by \cref{theoremlarge}, $\Out(W_\Gamma)$ is large. So we may assume $\Gamma$ contains exactly two connected components $\Gamma_1$ and $\Gamma_2$.

For $\{i,j\} = \{1,2\}$, each triple of vertices in $\Gamma_i$ must have at least 2 edges between them to prevent a STIL acting on $\Gamma_j$. Thus for a vertex $v \in \Gamma_i$, $\Gamma_i \backslash St(v)$ contains at most one vertex. This means $v \in \Gamma_i$ satisfies either:
\begin{itemize}
    \item $\Gamma \backslash St(v) = \Gamma_j$, where $\Gamma_j$ is the component which does not contain $v$. Then $v$ determines only an inner conjugation of $\Gamma$, which is a trivial partial conjugation in $\Out^0(W_\Gamma)$.
    \item $v$ is a star cut point such that $\Gamma \backslash St(v)$ has exactly two components; then there is a unique partial conjugation with acting vertex $v$, \newline $\chi_{v,\Gamma_j} \in \Out^0(W_\Gamma)$. In particular, this is the case for any acting vertex in a SIL. This is the only partial conjugation with acting vertex $v$ because if $w = \Gamma_i \backslash St(v)$, then $\chi_{v,w} \circ \chi_{v,\Gamma_j}$ is an inner conjugation, thus trivial in $\Out^0(W_\Gamma)$. Then $\chi_{v,w} = \chi_{v,\Gamma_j}^{-1}$.
\end{itemize}

See \cref{picturedisconabelian1} and \cref{picturedisconabelian2} for examples of graphs that satisfy these criteria. With this well-defined identification of each vertex in $\Gamma$ with a partial conjugation in $\Out^0(W_\Gamma)$ (which may be trivial), and since partial conjugations generate $\Out^0(W_\Gamma)$, the authors prove the existence of a surjective homomorphism \[\iota: W_{\Gamma_1} \times W_{\Gamma_2} \rightarrow \Out^0(W_\Gamma).\] Using the first isomorphism theorem, they prove that \[\ker (\iota) = Z(W_{\Gamma_1}) \times Z(\Gamma_2)\] \[\Rightarrow W_{\Gamma_1}/Z(W_{\Gamma_1}) \times W_{\Gamma_2}/Z(W_{\Gamma_2}) \cong \Out^0(W_\Gamma).\] 

Then each $W_{\Gamma_i}/Z(W_{\Gamma_i})$ is isomorphic to the right-angled Coxeter group with defining graph $\Gamma_i \backslash Z(\Gamma_i)$, and $\Out^0(W_\Gamma)$ is isomorphic to the right-angled Coxeter group with defining graph $\Gamma_1 \backslash Z(\Gamma_1) \sqcup \Gamma_2\backslash Z(\Gamma_2)$. Each of $W_{\Gamma_1}$ and $W_{\Gamma_2}$ is virtually abelian, because every triple of vertices in each component induces at least two edges. Therefore $\Out^0(W_\Gamma)$ is virtually abelian. 
\end{proof}

\begin{figure}
\centering
\begin{tikzpicture}
    \draw (1.5,3.5) ellipse (1.5 and 2);
    \node at (1.5,6) {$\Gamma_1$};
    \draw (5.5,3.5) ellipse (1.5 and 2);
    \node at (5.5,6) {$\Gamma_2$};
    \filldraw (2,4) circle (0.05) node[above]{$v_1$};
    \filldraw (2,2.5) circle (0.05) node[above]{$v_2$};
    \filldraw (1,3) circle (0.1) node[above left]{$K_{n_1}$};
    \draw (2,4) -- (1,3) -- (2,2.5);
    \filldraw (5,4) circle (0.05) node[above] {$w_1$};
    \filldraw (5,2.5) circle (0.05) node[above] {$w_2$};
    \draw (5,4) -- (6,3) -- (5,2.5);
    \filldraw (6,3) circle (0.1) node[above right]{$K_{n_2}$};
\end{tikzpicture}
\caption[Example 1 of a disconnected graph $\Gamma$ with $\Out^0(W_\Gamma)$ virtually abelian but not virtually $\Z$]{Example 1 of a disconnected graph $\Gamma$ with $\Out^0(W_\Gamma)$ virtually abelian but not virtually $\Z$. Here the SILs are $\{v_1,v_2|\Gamma_2\}$ and $\{w_1,w_2|\Gamma_1\}$. Let $|V(\Gamma_1) - 2| = n_1$ and $|V(\Gamma_2) -2| = n_2$.}
\label{picturedisconabelian1}
\end{figure}
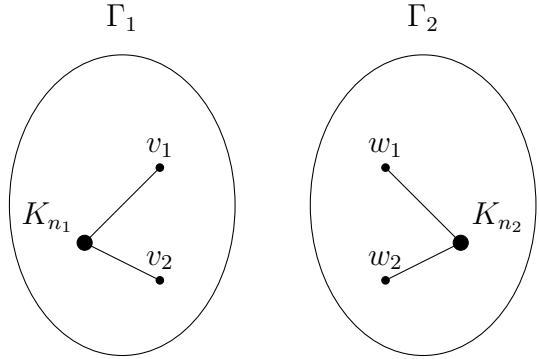

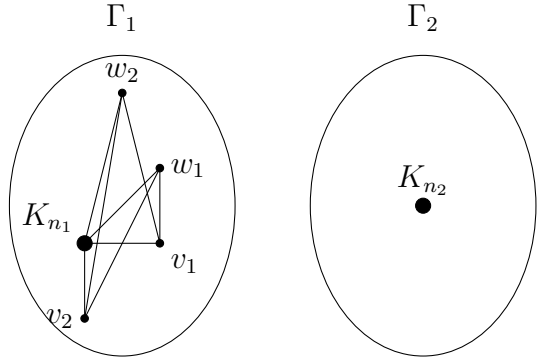
\begin{figure}
\centering
\begin{tikzpicture}
    \draw (1.5,3.5) ellipse (1.5 and 2);
    \node at (1.5,6) {$\Gamma_1$};
    \draw (5.5,3.5) ellipse (1.5 and 2);
    \node at (5.5,6) {$\Gamma_2$};
    \filldraw (2,3) circle (0.05) node[below right]{$v_1$};
    \filldraw (1,2) circle (0.05) node[left]{$v_2$};
    \filldraw (1,3) circle (0.1) node[above left]{$K_{n_1}$};
    \draw (1,2) -- (1,3) -- (2,3);
    \filldraw (2,4) circle (0.05) node[right] {$w_1$};
    \filldraw (1.5,5) circle (0.05) node[above] {$w_2$};
    \draw (2,4) -- (1,3) -- (1.5,5);
    \draw (1.5,5) -- (2,3) -- (2,4);
    \draw (1.5,5) -- (1,2) -- (2,4);
    \filldraw (5.5,3.5) circle (0.1) node[above]{$K_{n_2}$};
\end{tikzpicture}

\caption[Example 2 of a disconnected graph $\Gamma$ with $\Out^0(W_\Gamma)$ virtually abelian but not virtually $\Z$]{Example 2 of a disconnected graph $\Gamma$ with $\Out^0(W_\Gamma)$ virtually abelian but not virtually $\Z$. Here the SILs are $\{v_1,v_2|\Gamma_2\}$ and $\{w_1,w_2|\Gamma_2\}$. Let $|V(\Gamma_1) - 4| = n_1$ and $|V(\Gamma_2)| = n_2$. Note that even with multiple SILs in $\Gamma_1$, for each vertex $v \in \Gamma_1$, there is at most one vertex in $\Gamma_1 \backslash St(v)$.}
\label{picturedisconabelian2}
\end{figure}

\subsection{Applying these methods to our case}

The following proposition is the same result as \cref{theorem3.5}, only now for the disconnected case. 

\begin{proposition}\label{lemma6}
Let $W_\Gamma$ be a graph product of primary cyclic groups where $\Gamma$ is a disconnected graph. If $\Gamma$ contains two or more distinct Coxeter SILs, then $\Out^0(W_\Gamma)$ is not virtually $\Z$. 
\end{proposition}

\begin{proof}

This proof closely follows that of \cref{prop2.3}. If $\Gamma$ does not satisfy the characterization given above, then $\Out^0(W_\Gamma)$ is not virtually abelian, hence not virtually $\Z$. If $\Gamma$ contains a non-Coxeter SIL, then by \cref{theoremlarge}, $\Out(W_\Gamma)$ is large. So we may assume that $\Gamma$ has exactly two connected components $\Gamma_1$ and $\Gamma_2$, that every triple of vertices in $\Gamma_i$ has at least three edges, and that any SIL in $\Gamma$ is a Coxeter SIL.

First, let us consider the case where $\Gamma$ has at least one Coxeter SIL in each connected component: say $\{v_1,v_2|\Gamma_2\}$ is a SIL with $v_1,v_2 \in \Gamma_1$, and $\{w_1,w_2|\Gamma_2\}$ is a SIL with $w_1,w_2 \in \Gamma_2$. $v_1$ and $v_2$ are not adjacent by definition, but they must be adjacent to each other vertex in $\Gamma_1$, and each other pair in $\Gamma_1$ must be adjacent. Similarly, $\Gamma_2 \backslash \{w_1\}$ and $\Gamma_2 \backslash \{w_2\}$ must both be complete subgraphs. The center of $\Gamma_1$ is $Z(\Gamma_1) = \Gamma_1 \backslash \{v_1,v_2\}$, and $Z(\Gamma_2) = \Gamma_2 \backslash \{w_1,w_2\}$. See \cref{picturedisconabelian1} for a diagram. 

Now, we want to apply the surjection $\iota$ defined in the proof of \cref{prop2.3}. Although Susse and Sale assume that $W_\Gamma$ is a right-angled Coxeter graph, all of the conditions required to define $\iota$ still hold: each vertex in $\Gamma \backslash \{v_1,v_2,w_1,w_2\}$ defines only a trivial partial conjugation in $\Out^0(W_\Gamma)$. Each of $v_1, v_2, w_1$, and $w_2$ determines a unique, non-trivial partial conjugation of order 2; for example, the partial conjugation determined by $v_1$ is $\chi_{v_1,\Gamma_2} = \chi_{v_1,v_2}^{-1}$. So even though elements in $\Gamma \backslash \{v_1,v_2,w_1,w_2\}$ may have order larger than 2, the surjection \[\iota: W_{\Gamma_1} \times W_{\Gamma_2} \rightarrow \Out^0(W_\Gamma)\] holds, as does the isomorphism \[W_{\Gamma_1}/Z(W_{\Gamma_1}) \times W_{\Gamma_2}/Z(W_{\Gamma_2}) \cong \Out^0(W_\Gamma).\] In the case that there are no other Coxeter SILs in $\Gamma$, we get \[\langle v_1, v_2 \rangle \times \langle w_1, w_2 \rangle \cong \Out^0(W_\Gamma)\] \[\Rightarrow D_\infty \times D_\infty \cong \Out^0(W_\Gamma),\] which is not virtually $\Z$. If there are additional Coxeter SILs in $\Gamma$, then $\langle v_1,v_2 \rangle \subseteq W_{\Gamma_1} / Z(W_{\Gamma_1})$ and $\langle w_1,w_2 \rangle \subseteq W_{\Gamma_2} / Z(W_{\Gamma_2})$, thus $D_\infty \times D_\infty$ is a subgroup of $\Out^0(W_\Gamma)$.

Now we consider the case where $\Gamma = \Gamma_1 \sqcup \Gamma_2$ contains at least two SILs in one of the $\Gamma_i$ and no SILs in the other component, as in \cref{picturedisconabelian2}. Suppose, without loss of generality, that $\{v_1,v_2|\Gamma_2\}$ and $\{w_1,w_2|\Gamma_2\}$ are SILs with acting vertices in $\Gamma_1$. If the set $\{v_1,v_2,w_1,w_2\}$ contains only three distinct vertices, say $v_1 = w_1$, then the triple $v_1,v_2$ and $w_2$ has one edge between $v_2$ and $w_2$, forming a STIL. Thus we assume that the set $\{v_1,v_2,w_1,w_2\}$ contains four distinct vertices.

We assume first that there are no other Coxeter SILs with acting vertices in $\Gamma_1$. Then $W_{\Gamma_1}/Z(W_{\Gamma_1}) = \langle v_1, v_2 \rangle \times \langle w_1, w_2 \rangle$, because both $v_1$ and $v_2$ commute with $w_1$ and $w_2$. No SILs with acting vertices in $\Gamma_2$ also implies that $Z(W_{\Gamma_2}) = W_{\Gamma_2}$. Applying $\iota$ again, \[\langle v_1, v_2 \rangle \times \langle w_1, w_2 \rangle \times 1 \cong \Out^0(W_\Gamma)\] \[\Rightarrow D_\infty \times D_\infty \cong \Out^0(W_\Gamma).\] As with the first case, if there are more than 2 Coxeter SILs with acting vertices in $\Gamma_1$, $D_\infty \times D_\infty$ is still a subgroup of $\Out^0(W_\Gamma)$, and the result holds.

\end{proof}

Here is a lemma that is useful for both the disconnected and connected cases, and proves that the conditions stated in our main result, \cref{theoremmain} in Chapter 6, agree with the conditions stated in \cref{theoremlarge} for $\Out(W_\Gamma)$ to be virtually abelian. 

\begin{lemma}\label{lemma4}
Let $W_\Gamma$ be a graph product of primary cyclic groups, where $\Gamma$ is an arbitrary graph. If $\Gamma$ contains exactly one SIL, then $\Gamma$ cannot contain a STIL. 
\end{lemma}

\begin{proof}
We prove the statement by contrapositive: suppose $\Gamma$ contains the STIL $\{v_1,v_2,v_3|Z\}$ as in \cref{picturelemma4}. Since the STIL has at most one edge between the vertices $v_1,v_2$ and $v_3$, we suppose, without loss of generality, that there is an edge between $v_2$ and $v_3$. We will leverage this to prove $\Gamma$ contains two SILs. If there are no edges in the subgraph induced by $v_1,v_2$, and $v_3$, then we will find three SILs. If $\Gamma$ is disconnected, it is possible that there is no edge connecting the components $L_1 \cap L_2 \cap L_3$ and $Z$, but the result still holds. 

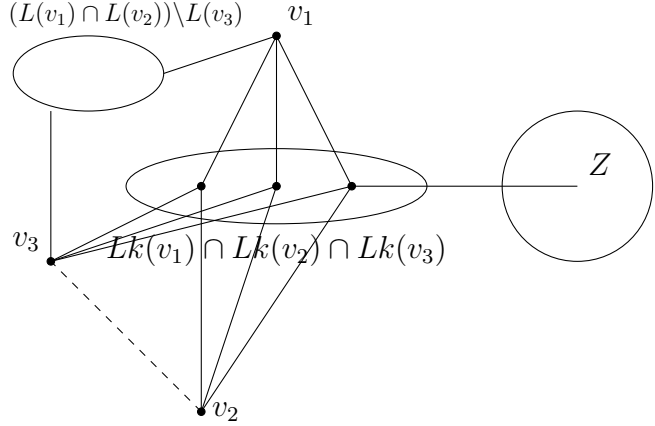
\begin{figure}
\centering
\begin{tikzpicture}
\draw (3,3) ellipse (2 and 0.5);
\node[below] at (3,2.5) {$Lk(v_1) \cap Lk(v_2) \cap Lk(v_3)$};
\filldraw (2,3) circle (0.05);
\filldraw (3,3) circle (0.05);
\filldraw (4,3) circle (0.05);
\draw (7,3) circle (1) node[above right]{$Z$};
\draw (4,3) -- (7,3);
\filldraw (3,5) circle (0.05) node[above right]{$v_1$};
\draw (2,3) -- (3,5);
\draw (3,3) -- (3,5);
\draw (4,3) -- (3,5);
\filldraw (2,0) circle (0.05) node[right]{$v_2$};
\draw (2,3) -- (2,0);
\draw (3,3) -- (2,0);
\draw (4,3) -- (2,0);
\filldraw (0,2) circle (0.05) node[above left]{$v_3$};
\draw (2,3) -- (0,2);
\draw (3,3) -- (0,2);
\draw (4,3) -- (0,2);
\draw[dashed] (2,0) -- (0,2);
\draw (0.5,4.5) ellipse (1 and 0.5);
\node[above][scale=0.8] at (1,5) {$(L(v_1) \cap L(v_2)) \backslash L(v_3)$};
\draw (3,5) -- (1.5,4.5);
\draw (0,2) -- (0,4);
\end{tikzpicture}
    \caption{STIL for proof of \cref{lemma4}}
    \label{picturelemma4}
\end{figure}

Suppose there is no SIL $\{v_1, v_2|Z\}$. Then there must be the paths $v_1 \rightarrow Z$ and $v_2 \rightarrow Z$, both passing through $(L(v_1)\cap L(v_2)) \backslash L(v_3)$. But then $v_1, v_2$, and $Z$ would be in the same connected component of $\Gamma \backslash (L(v_1) \cap L(v_2) \cap L(v_3))$, contradicting the assumption of a STIL $\{v_1,v_2,v_3|Z\}$. Therefore $\{v_1,v_2|Z\}$ must form a SIL. Similarly, $\{v_1,v_3|Z\}$ must form a SIL because there can be no paths $v_1 \rightarrow Z, v_2 \rightarrow Z$ in $\Gamma \backslash (L(v_1) \cap L(v_2) \cap L(v_3))$. 

Therefore, the existence of a STIL in $\Gamma$ implies the existence of at least two SILs in $\Gamma$.
\end{proof}

\cref{lemma6} proves that a disconnected graph $\Gamma$ must have a unique Coxeter SIL in order for $\Out(W_\Gamma)$ to be virtually $\Z$. We conclude this chapter by showing that this is a sufficient condition for $\Out(W_\Gamma)$ to be virtually $\Z$, and the proof describes the family of graphs satisfying these conditions. 

\begin{proposition}\label{lemma5}

Let $W_\Gamma$ be a graph product of primary cyclic groups where $\Gamma$ is a disconnected graph. If $\Gamma$ has a unique Coxeter SIL $\{v_1,v_2| C\}$ and no non-Coxeter SIL, then $\Out^0(W_\Gamma) \cong D_\infty$.
\end{proposition}

\begin{proof}

We again follow the structure of the proof of \cref{prop2.3}. By \cref{lemma4}, since $\Gamma$ has exactly one SIL, there can be no STIL or FSIL in $\Gamma$. As noted in the proof of \cref{prop2.3}, $\Gamma$ must have exactly 2 connected components $\Gamma_1$ and $\Gamma_2$. Since there are two components, the SIL must be acting on an entire connected component. Suppose, without loss of generality, that the Coxeter SIL is $\{v_1, v_2|\Gamma_2\}$, where $v_1$ and $v_2$ are in $\Gamma_1$. 

$v_1$ is not adjacent to $v_2$, but every other pair in $\Gamma_1$ must be adjacent, otherwise we get a second SIL acting on $\Gamma_2$. $\Gamma_2$ must be a complete subgraph, otherwise we get a second SIL acting on $\Gamma_1$. Therefore $Z(\Gamma_1) = \Gamma_1 \backslash \{v_1,v_2\}$ and $Z(\Gamma_2) = \Gamma_2$. See \cref{picturedisconcyclic} for a diagram.

In \cref{lemma6} we saw that the surjection $\iota$ defined in \cref{prop2.3} still holds when $\Gamma$ contains only Coxeter SILs. We apply the isomorphism determined by $\iota$ and get  \[W_{\Gamma_1}/Z(W_{\Gamma_1}) \times W_{\Gamma_2}/Z(W_{\Gamma_2}) \cong \Out^0(W_\Gamma)\] \[\Rightarrow \langle v_1, v_2 \rangle \times 1 \cong \Out^0(W_\Gamma)\] \[\Rightarrow D_\infty \cong \Out^0(W_\Gamma).\] In particular, $\Out^0(W_\Gamma)$ is virtually $\Z$.
\end{proof}

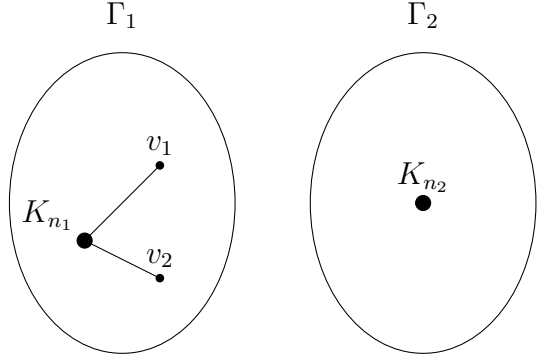
\begin{figure}
\centering
\begin{tikzpicture}
    \draw (1.5,3.5) ellipse (1.5 and 2);
    \node at (1.5,6) {$\Gamma_1$};
    \draw (5.5,3.5) ellipse (1.5 and 2);
    \node at (5.5,6) {$\Gamma_2$};
    \filldraw (2,4) circle (0.05) node[above]{$v_1$};
    \filldraw (2,2.5) circle (0.05) node[above]{$v_2$};
    \filldraw (1,3) circle (0.1) node[above left]{$K_{n_1}$};
    \draw (2,4) -- (1,3) -- (2,2.5);
    \filldraw (5.5,3.5) circle (0.1) node[above]{$K_{n_2}$};
\end{tikzpicture}
\caption[Disconnected graph $\Gamma$ with $\Out^0(W_\Gamma)$ virtually $\Z$]{A characterization of all disconnected graphs $\Gamma$ with $\Out^0(W_\Gamma)$ virtually $\Z$. Let $|V(\Gamma_1) - 2| = n_1$ and $|V(\Gamma_2)| = n_2$.}
\label{picturedisconcyclic}
\end{figure}

\section{Proving our main result}

This chapter covers a few more necessary lemmas and concludes with a proof of \cref{theoremmain}.

\newpage

\begin{lemma}[Lemma 1.4 \cite{sale2017outerautomorphismgroupsrightangled} ]\label{lemma1.4}

Suppose $W_\Gamma$ is a graph product of directly indecomposable cyclic groups and let $x$ and $y$ be distinct vertices in $\Gamma$. Two partial conjugations $\chi_{x,C}$ and $\chi_{y,D} $ in $\Out(W_\Gamma)$ do not commute if and only if there is a SIL $\{x,y|z\}$ and one of the following holds:
    \begin{enumerate}
        \item $z \in C = D$;
        \item $x \in D$ and $z \in C$;
        \item $y \in C$ and $z \in D$;
        \item $x \in D$ and $y \in C$.
    \end{enumerate}
\end{lemma}

The next lemma shows that the condition of $\Gamma$ having exactly one SIL is very strong. A similar result was shown in the proof of \cref{prop2.3} for the disconnected case. Again, this result implies that the identification of a SIL with a pair of non-trivial partial conjugations in $\Out^0(W_\Gamma)$ is unique when $\Gamma$ has exactly one Coxeter SIL, a fact that we will prove explicitly for the general case in \cref{theoremmain}.

\begin{lemma}\label{lemma7}

Let $W_\Gamma$ be a graph product of primary cyclic groups, where $\Gamma$ is a connected graph. If $\Gamma$ contains exactly one SIL $\{v_1,v_2|Z\}$, then each of $\Gamma \backslash St(v_1)$ or $\Gamma \backslash St(v_2)$ has exactly two connected components.
\end{lemma}

\begin{proof}

We prove the statement by contrapositive: suppose $\Gamma \backslash St(v_1)$ has 3 connected components, $Z, C, D$, labeled so that $v_2 \in C$, as in \cref{picturelemma7}.

\begin{figure}
    \centering
    \begin{tikzpicture}[scale=0.6]
    \filldraw (7,13) circle (0.05) node[above]{$v_1$};
    \filldraw (7,3) circle (0.05) node[below]{$v_2$};
    \draw (7,13) -- (6,8) -- (7,3);
    \draw (7,13) -- (7,8) -- (7,3);
    \draw (7,13) -- (8,8) -- (7,3);
    \filldraw (6,8) circle (0.05) node[above left]{$a$};
    \filldraw (7,8) circle (0.05);
    \filldraw (8,8) circle (0.05);
    \draw (7,8) ellipse (3 and 0.5);
    \node[below right] at (7.5,7.5) {$Lk(v_1) \cap Lk(v_2)$};
    \draw (13,10) circle (2) node[above]{$D$};
    \draw (7,13) -- (12,11);
    \draw (7,4) ellipse (3 and 2);
    \node[right] at (7,3.5) {$C$};
    \draw (6,8) -- (3,8);
    \draw (2,8) ellipse (1 and 3) node{$Z$};
    \filldraw (3,8) circle (0.05) node[left]{$z$};
    \draw (7,13) arc (90:270:3.5 and 5);
    \filldraw (5,12.10326) circle (0.05) node[above]{$b$};
    \end{tikzpicture}
    \caption{SIL for \cref{lemma7}, where $\Gamma \backslash St(v_1)$ has 3 connected components.}
    \label{picturelemma7}
\end{figure}
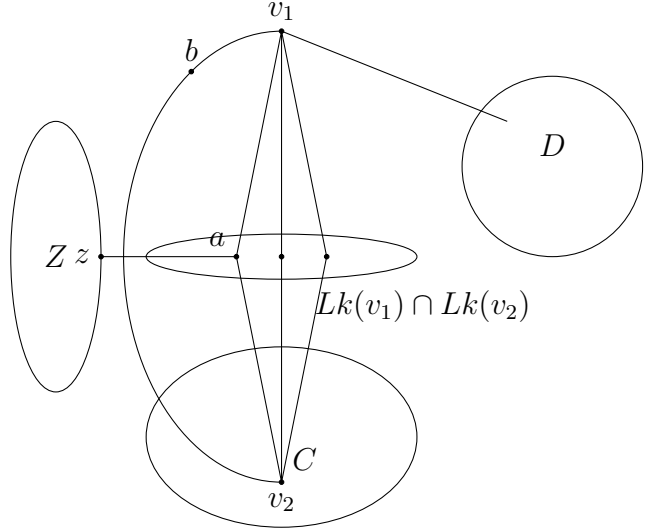

Suppose no path $\gamma$ between $v_1$ and $v_2$ exists which does not pass through $Lk(v_1) \cap Lk(v_2)$. Then consider $z \in Z$ such that $Lk(z) \cap Lk(v_1) \cap Lk(v_2) \ne \emptyset$. $v_1$ and $v_2$ are in separate components of $\Gamma \backslash (Lk(v_1) \cap Lk(v_2))$, so $\{v_1,v_2,z\}$ forms an FSIL, which of course implies that $\Gamma$ has more than one SIL. Therefore such a path $\gamma$ must exist. $\gamma$ has length $\ge 2$ as a condition of the SIL; it cannot pass through $Z$ or $C$, because then these components would not be disconnected from one another. The only path between any two of $C,D$, and $Z$ must pass through either $Lk(v_1) \cap Lk(v_2)$ or $v_1$. Let $a$ be a vertex in $Lk(v_1) \cap Lk(v_2)$ and let $b$ be the point in $\gamma$ adjacent to $v_1$. Then $\{a,b |D\}$ forms a SIL, so $\Gamma$ has two distinct SILs. 

Similarly, if $\Gamma \backslash St(v_2)$ has at least three connected components, then $\Gamma$ has more than one SIL. 
\end{proof}

The proof of our main theorem, \cref{theoremmain}, has a similar structure to the proof of Proposition 2.2 in \cite{sale2017outerautomorphismgroupsrightangled}.

\begin{proposition}[Proposition 2.2 \cite{sale2017outerautomorphismgroupsrightangled}] \label{prop2.2}
Suppose that $W_\Gamma$ is a right-angled Coxeter group defined by a connected graph $\Gamma$ containing no STIL or FSIL. Then the commutator subgroup $(\Out^0(W_\Gamma))'$ is  abelian.
\end{proposition}

In particular, we will see that the commutator subgroup is virtually cyclic when $\Gamma$ satisfies our conditions. 

\begin{definition}
    Let $G$ be an arbitrary group. For elements $h$ and $g$ in $G$, we define their \textbf{commutator} as $[g,h] = ghg^{-1}h^{-1}$. The \textbf{commutator subgroup} (also called the \textbf{derived subgroup}) of $G$, denoted $G'$, is the subgroup generated by all commutators of elements in $G$. Note that $G'$ is normal in $G$.
\end{definition}

Finally, we have a necessary and sufficient condition for $\Out(W_\Gamma)$ to be virtually $\Z$.

\begin{theorem}\label{theoremmain}

Let $W_\Gamma$ be graph product of primary cyclic groups. $\Out(W_\Gamma)$ is virtually $\mathbb{Z}$ if and only if $\Gamma$ contains a unique Coxeter SIL and no non-Coxeter SIL. 

\end{theorem}

\begin{proof}

We prove the $(\Rightarrow)$ direction by contrapositive.

If $\Gamma$ has no Coxeter SIL, then \cref{theoremfinite} shows that $\Out(W_\Gamma)$ is finite. Applying \cref{theorem3.5} for the connected case and \cref{lemma6} for the disconnected case, if $\Gamma$ contains two distinct Coxeter SILs, then $\Out(W_\Gamma)$ is not virtually $\Z$. If $\Gamma$ contains a non-Coxeter SIL, then by \cref{theoremlarge}, $\Out(W_\Gamma)$ is large. Therefore if either of these conditions fails, $\Out(W_\Gamma)$ is not virtually $\Z$, proving the ($\Rightarrow$) direction. 

Now we prove the $(\Leftarrow)$ direction directly. Apply \cref{lemma5} for the case where $\Gamma$ is disconnected; then $\Out^0(W_\Gamma) \cong D_\infty$, which is virtually $\Z$.

Now assume $\Gamma$ is connected and let $\{v_1,v_2|Z\}$ be the unique Coxeter SIL. By \cref{lemma4}, since $\Gamma$ has exactly one Coxeter SIL, there can be no STIL or FSIL in $\Gamma$. \cref{lemma2.2} shows that this SIL corresponds to the pair of partial conjugations $\chi_{v_1,Z}$ and $\chi_{v_2,Z}$ in $\Out^0(W_\Gamma)$. By \cref{lemma7}, since $\Gamma \backslash St(v_1)$ contains exactly two connected components, $\chi_{v_1,Z}$ is the unique partial conjugation whose acting vertex is $v_1$ in $\Out^0(W)$: Let $D = \Gamma \backslash (St(v_1) \cup Z)$ be the other connected component that $v_1$ acts on by partial conjugation. Then $\chi_{v_1,Z} \circ \chi_{v_1,D} = 1$ in $\Out^0(W_\Gamma)$, because this is an inner conjugation of $\Gamma$ by $v_1$, and inner conjugations are trivial in $\Out^0(W_\Gamma)$. Thus $\chi_{v_1,Z} = (\chi_{v_1,D})^{-1} = \chi_{v_1,D}$ in $\Out^0(W_\Gamma)$, using the fact that $v_1$ and $v_2$ have order 2. Similarly, $\chi_{v_2,Z}$ is the unique partial conjugation in $\Out^0(W_\Gamma)$ whose acting vertex is $v_2$.

Now consider the commutator $[\chi_{v_1,Z}, \chi_{v_2,Z}]$. By conjugating this commutator by $\chi_{v_1,Z}$, we get \[\chi_{v_1,Z} \circ (\chi_{v_1,Z} \circ \chi_{v_2,Z} \circ \chi_{v_1,Z} \circ \chi_{v_2,Z}) \circ \chi_{v_1,Z}\] \[= [\chi_{v_2,Z}, \chi_{v_1,Z}] = [\chi_{v_1,Z}, \chi_{v_2,Z}]^{-1}.\] On the other hand, if we conjugate by $\chi_{v_2,Z}$, we get \[\chi_{v_2,Z} \circ (\chi_{v_1,Z} \circ \chi_{v_2,Z} \circ \chi_{v_1,Z} \circ \chi_{v_2,Z}) \circ \chi_{v_2,Z}\] \[= [\chi_{v_2,Z}, \chi_{v_1,Z}] = [\chi_{v_1,Z}, \chi_{v_2,Z}]^{-1}.\] This is analogous to the proof of Lemma 2.5 from \cite{sale2017outerautomorphismgroupsrightangled}, but these are the only two cases, as $v_1$ and $v_2$ are each the acting vertex in a unique partial conjugation in $\Out^0(W)$.

Let $b \in \Gamma \backslash \{v_1,v_2\}$ be a star cut point and let $\chi_{b,C}$ be a partial conjugation in $\Out^0(W_\Gamma)$. By \cref{lemma1.4}, $\chi_{b,C}$ commutes with $\chi_{v_i,Z}$ (for $i \in \{1,2\}$) if and only if there is a SIL with acting vertices $b$ and $v_i$ satisfying certain conditions. However, since $\{v_1,v_2|Z\}$ is the unique SIL in $\Gamma$, and $b \notin \{v_1,v_2\}$, then the partial conjugations $\chi_{b,C}$ and $\chi_{v_i,Z}$ commute. As a result, $[\chi_{b,C},\chi_{v_i,Z}] = 1$ for $i \in \{1,2\}$. Similarly, for any other partial conjugation $\chi_{b',C'}$ in $\Out^0(W_\Gamma)$, $[\chi_{b,C}, \chi_{b',C'}] = 1$, so the only non-trivial commutator in $(\Out^0(W_\Gamma))'$ is $[\chi_{v_1,Z}, \chi_{v_2,Z}]$. When we commute $[\chi_{v_1,Z}, \chi_{v_2,Z}]$ by $\chi_{b,C}$, we have \[\chi_{b,C} \circ [\chi_{v_1,Z}, \chi_{v_2,Z}] \circ \chi_{b,C}^{-1}\] \[= \chi_{b,C} \circ \chi_{b,c}^{-1} \circ [\chi_{v_1,Z}, \chi_{v_2,Z}]\] \[= [\chi_{v_1,Z}, \chi_{v_2,Z}],\] so the subgroup $\langle [\chi_{v_1,Z}, \chi_{v_2,Z}] \rangle$ is closed in $\Out^0(W_\Gamma)$ under conjugation.

Thus, we conclude that $\chi_1[\chi_2,\chi_3]\chi_1 \in \langle[\chi_{v_1,Z}, \chi_{v_2,Z}] \rangle$ for all partial conjugations $\chi_1, \chi_2, \chi_3 \in \Out^0(W_\Gamma)$. Therefore the commutator subgroup is $(\Out^0(W_\Gamma))' \cong \langle[\chi_{v_1,Z}, \chi_{v_2,Z}] \rangle$. $[\chi_{v_1,Z}, \chi_{v_2,Z}]^2 = (\chi_{v_1,Z}\chi_{v_2,Z})^4$, so $[\chi_{v_1,Z}, \chi_{v_2,Z}]$ is an infinite order element, and $(\Out^0(W_\Gamma))' \cong \Z$.

The abelianization $\Out^0(W_\Gamma)/(\Out^0(W_\Gamma))'$ is finite because $\Out^0(W_\Gamma)$ is finitely generated by finite order elements, so $|\Out^0(W_\Gamma):(\Out^0(W_\Gamma))'|$ is finite. Then we have a finite index normal subgroup isomorphic to $\Z$ in $\Out^0(W_\Gamma)$, thus $\Out^0(W_\Gamma)$ is virtually $\Z$, and hence $\Out(W_\Gamma)$ is virtually $\Z$.
\end{proof}

Now, we describe connected graphs which satisfy the conditions of \cref{theoremmain}.

\begin{example}
The condition that $\Gamma$ must have exactly one Coxeter SIL places strong restrictions on its structure. First, we have a Coxeter SIL $\{v_1,v_2|C\}$. There can be at most one path of length 1 starting at $v_1$, and similarly for $v_2$. If we had two or more paths of length 1 starting at $v_1$, then $v_1$ would be in the intersection of the links of the endpoints of these paths, and the endpoints would form a second SIL. If there is a path of length 1 at either $v_1$ or $v_2$, then there is exactly one path of length at least 2 between $v_1$ and $v_2$ which does not pass through $Lk(v_1) \cap Lk(v_2)$, as in \cref{picturethmex1}. We cannot have two or more paths between $v_1$ and $v_2$ outside of $Lk(v_1) \cap Lk(v_2)$ in this case, otherwise two vertices on these paths would form a second SIL acting on the endpoint of the path of length 1. 

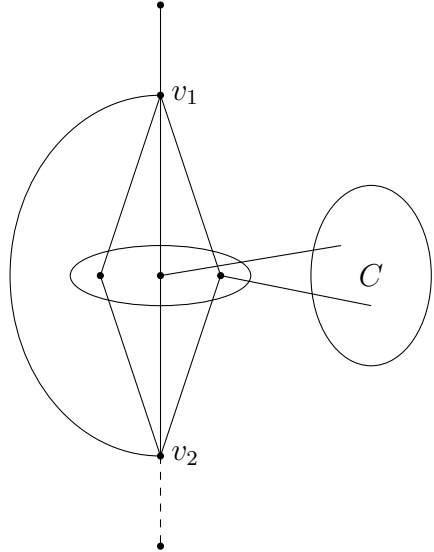
\begin{figure}
    \centering
    \begin{tikzpicture}[scale=0.8]
        \filldraw (1,7) circle (0.05) node[right]{$v_1$};
        \filldraw (1,1) circle (0.05) node[right]{$v_2$};
        \filldraw (0,4) circle (0.05);
        \filldraw (1,4) circle (0.05);
        \filldraw (2,4) circle (0.05);
        \draw (1,4) ellipse (1.5 and 0.5);
        \draw (1,7) -- (0,4) -- (1,1);
        \draw (1,7) -- (1,4) -- (1,1);
        \draw (1,7) -- (2,4) -- (1,1);        
        \draw (1,7) arc (90:270:2.5 and 3);
        \draw (4.5,4) ellipse (1 and 1.5) node{$C$};
        \draw (1,4) -- (4,4.5);
        \draw (2,4) -- (4.5,3.5);
        \filldraw (1,8.5) circle (0.05);
        \draw (1,7) -- (1,8.5);
        \filldraw (1,-0.5) circle (0.05);
        \draw[dashed] (1,1) -- (1,-0.5);
    \end{tikzpicture}
    \caption[One type of connected graph $\Gamma$ with exactly one Coxeter SIL.]{One type of connected graph $\Gamma$ with exactly one Coxeter SIL. $v_1$ is the endpoint of a path of length 1. The dashed path may or may not be present.}
    \label{picturethmex1}
\end{figure}

If there is no path of length 1 starting at $v_1$ or $v_2$, then there is at least one path connecting $v_1$ and $v_2$ which does not pass through the intersection of links, and there is no maximum to how many such paths exist in $\Gamma$. See \cref{picturethmex2}. 

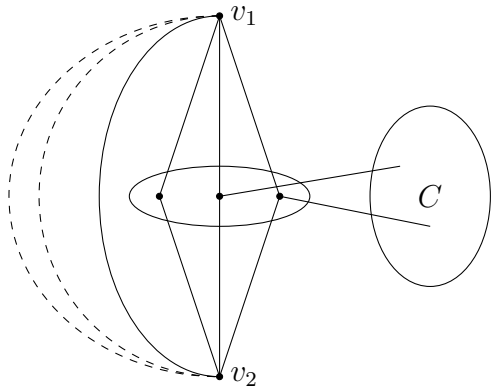
\begin{figure}
    \centering
    \begin{tikzpicture}[scale=0.8]
        \filldraw (1,7) circle (0.05) node[right]{$v_1$};
        \filldraw (1,1) circle (0.05) node[right]{$v_2$};
        \filldraw (0,4) circle (0.05);
        \filldraw (1,4) circle (0.05);
        \filldraw (2,4) circle (0.05);
        \draw (1,4) ellipse (1.5 and 0.5);
        \draw (1,7) -- (0,4) -- (1,1);
        \draw (1,7) -- (1,4) -- (1,1);
        \draw (1,7) -- (2,4) -- (1,1);        
        \draw (1,7) arc (90:270:2 and 3);
        \draw[dashed] (1,7) arc (90:270:3);
        \draw[dashed] (1,7) arc (90:270:3.5 and 3);
        \draw (4.5,4) ellipse (1 and 1.5) node{$C$};
        \draw (1,4) -- (4,4.5);
        \draw (2,4) -- (4.5,3.5);
    \end{tikzpicture}
    \caption[The second type of connected graph $\Gamma$ with exactly one Coxeter SIL.]{The second type of connected graph $\Gamma$ with exactly one Coxeter SIL. There is no maximum on the number of paths between $v_1$ and $v_2$, but there must be at least one outside $Lk(v_1) \cap Lk(v_2)$.}
    \label{picturethmex2}
\end{figure}

The connected component $C$ also cannot contain another SIL. Note there is no restriction on the value of the order map \textbf{m} on any vertices besides $v_1$ and $v_2$. 

\end{example}

\begin{example}\label{examplemainthm}
We give a few concrete examples to show how removing or rearranging one or two edges is the difference between $\Out(W_\Gamma)$ being virtually $\Z$ or not. As in \cref{pictureconcreteex1}, suppose we have a right-angled Coxeter group $W_\Gamma$ and that $\Gamma$ has a SIL $\{v_1,v_2|C\}$, where $C$ is the subgraph induced by $d,e$, and $f$. Since $C$ is a complete subgraph, we do not get a second SIL in $\Gamma$, so $\Out(W_\Gamma)$ is virtually $\Z$. We can verify this concretely: the star cut points of $\Gamma$ are $v_1,v_2$, and $c$, each acting in two non-trivial partial conjugations. We construct $\mathcal{P}^0$, the generating set of $\Out^0(W_\Gamma)$, by assigning the vertex $a$ a 1 in the ordering of the vertices, and the other vertices get an arbitrary ordering. In this way, we do not include $\chi_{v_1,a}, \chi_{v_2,a}$, or $\chi_{c,a}$ in $\mathcal{P}^0$. Then \[\mathcal{P}^0 = \{\chi_{v_1,C}, \chi_{v_2,C}, \chi_{c, \{e,f\}}\}.\] Of these partial conjugations, only $\chi_{v_1,C}$ and $\chi_{v_2,C}$ do not commute, by \cref{lemma1.4}. Thus \[\Out^0(W_\Gamma) \cong (\Z/2\Z * \Z/2\Z) \times \Z/2\Z \cong D_\infty \times \Z/2\Z,\] which is virtually $\Z$. 

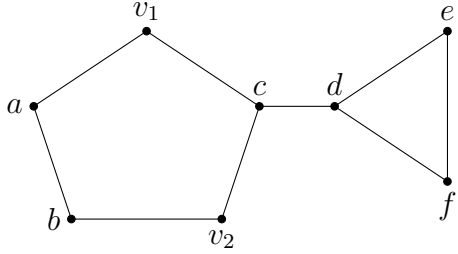
\begin{figure}
    \centering
    \begin{tikzpicture}
        \filldraw (1.5,4) circle (0.05) node[above]{$v_1$};
        \filldraw (3,3) circle (0.05) node[above]{$c$};
        \filldraw (0, 3) circle (0.05) node[left]{$a$};
        \filldraw (0.5,1.5) circle (0.05) node[left]{$b$};
        \filldraw (2.5,1.5) circle (0.05) node[below]{$v_2$};
        \draw (1.5,4) -- (3,3) -- (2.5,1.5) -- (0.5,1.5) -- (0,3) -- (1.5,4);
        \draw (3,3) -- (4,3) -- (5.5,4) -- (5.5,2) -- (4,3);
        \filldraw (4,3) circle (0.05) node[above]{$d$};
        \filldraw (5.5,4) circle (0.05) node[above]{$e$};
        \filldraw (5.5,2) circle (0.05) node[below]{$f$};
    \end{tikzpicture}
    \caption[Example 1 of a graph satisfying the conditions of \cref{theoremmain}.]{Example 1 of a graph satisfying the conditions of \cref{theoremmain}. We assume each vertex is an element of order 2 in $W_\Gamma$.}
    \label{pictureconcreteex1}
\end{figure}

Now suppose we remove the edge between the vertices $d$ and $f$, as in \cref{pictureconcreteex2}. Even though $C$ now has one fewer edge, still no pair of vertices in $C$ has an intersection of links that separates them from the subgraph containing $v_1$ and $v_2$. Therefore we have no additional SILs and $\Out(W_\Gamma)$ still satisfies the condition of \cref{theoremmain}. Again, we can check this concretely: the star cut points of $\Gamma$ are $v_1,v_2,c$, and $d$, each acting in two non-trivial partial conjugations. As in the previous case, we construct $\mathcal{P}^0$ by assigning the vertex $a$ a 1 in the ordering of the vertices. Then \[\mathcal{P}^0 = \{\chi_{v_1,C}, \chi_{v_2,C}, \chi_{c, \{e,f\}}, \chi_{d, \{f\}}\}.\] Because we added no new SILs, we still have only one pair of partial conjugations which do not commute, $\chi_{v_1,C}$ and $\chi_{v_2,C}$. Thus \[\Out^0(W_\Gamma) \cong D_\infty \times \Z/2\Z \times \Z/2\Z.\] 

\begin{figure}
    \centering
    \begin{tikzpicture}
        \filldraw (1.5,4) circle (0.05) node[above]{$v_1$};
        \filldraw (3,3) circle (0.05) node[above]{$c$};
        \filldraw (0, 3) circle (0.05) node[left]{$a$};
        \filldraw (0.5,1.5) circle (0.05) node[left]{$b$};
        \filldraw (2.5,1.5) circle (0.05) node[below]{$v_2$};
        \draw (1.5,4) -- (3,3) -- (2.5,1.5) -- (0.5,1.5) -- (0,3) -- (1.5,4);
        \draw (3,3) -- (4,3) -- (5.5,4) -- (5.5,2);
        \filldraw (4,3) circle (0.05) node[above]{$d$};
        \filldraw (5.5,4) circle (0.05) node[above]{$e$};
        \filldraw (5.5,2) circle (0.05) node[below]{$f$};
    \end{tikzpicture}
    \caption{Example 2 of a graph satisfying the conditions of \cref{theoremmain}.}
    \label{pictureconcreteex2}
\end{figure}
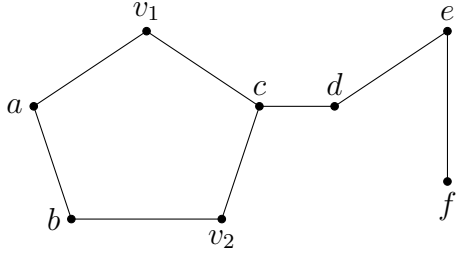

Finally, suppose we return to the first case and remove the edge between the vertices $e$ and $f$, as in \cref{pictureconcreteex3}. Now $\{e,f|D\}$ forms a second distinct SIL in $\Gamma$, where $D = \{v_1,a,b,v_2,c\}$. $\Gamma$ contains no FSIL or STIL, so $\Out(W_\Gamma)$ is virtually abelian but not virtually $\Z$. In this case, the star cut points are $v_1,v_2,c,e$, and $f$; each of these act in two non-trivial partial conjugations, except for $c$, which acts in three. We construct $\mathcal{P}^0$ by assigning the vertex $a$ a 1 in the ordering of the vertices. Then \[\mathcal{P}^0 = \{\chi_{v_1,C}, \chi_{v_2,C}, \chi_{c,\{e\}}, \chi_{c, \{f\}}, \chi_{e,\{f\}}, \chi_{f, \{e\}}\}.\] In $\Out^0(W_\Gamma)$, $\chi_{e,\{f\}} = \chi_{e,D}^{-1} = \chi_{e,D}$, because $e$ is an order 2 element. Similarly, $\chi_{f,\{e\}} = \chi_{f,D}$ in $\Out^0(W_\Gamma)$. Now that we have two SILs, we have two pairs of non-commuting partial conjugations, $\chi_{v_1,C}, \chi_{v_2,C}$ and $\chi_{e,D}, \chi_{f,D}$, each generating a $D_\infty$ in $\Out^0(W_\Gamma)$. Thus \[\Out^0(W_\Gamma) \cong D_\infty \times D_\infty \times \Z/2\Z \times \Z/2\Z.\]

This shows how both the number and placement of edges in $\Gamma$ determine the structure of $\Out(W_\Gamma)$. 

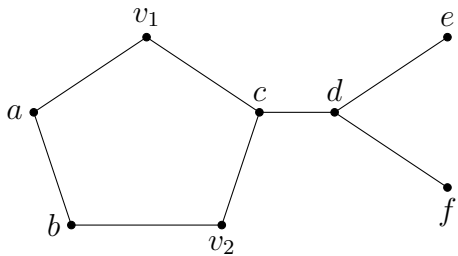
\begin{figure}
    \centering
    \begin{tikzpicture}
        \filldraw (1.5,4) circle (0.05) node[above]{$v_1$};
        \filldraw (3,3) circle (0.05) node[above]{$c$};
        \filldraw (0, 3) circle (0.05) node[left]{$a$};
        \filldraw (0.5,1.5) circle (0.05) node[left]{$b$};
        \filldraw (2.5,1.5) circle (0.05) node[below]{$v_2$};
        \draw (1.5,4) -- (3,3) -- (2.5,1.5) -- (0.5,1.5) -- (0,3) -- (1.5,4);
        \draw (3,3) -- (4,3) -- (5.5,4);
        \draw (4,3) -- (5.5,2);
        \filldraw (4,3) circle (0.05) node[above]{$d$};
        \filldraw (5.5,4) circle (0.05) node[above]{$e$};
        \filldraw (5.5,2) circle (0.05) node[below]{$f$};
    \end{tikzpicture}
    \caption[An example of a graph which does not satisfy the conditions of \cref{theoremmain}.]{An example of a graph which does not satisfy the conditions of \cref{theoremmain}. Even though the vertices $d,e$, and $f$ induce the same number of edges as in \cref{pictureconcreteex2}, this graph contains two distinct SILs.}
    \label{pictureconcreteex3}
\end{figure}

\end{example}

Finally, we encapsulate \cref{theoremfinite}, \cref{theoremlarge}, and \cref{theoremmain} in one statement. 

\begin{theorem}
    Let $\Gamma$ be a graph product of primary cyclic groups.
    \begin{enumerate}
        \item $\Gamma$ contains no SIL exactly when $\Out(W_\Gamma)$ is finite.
        \item $\Gamma$ contains a SIL exactly when $\Out(W_\Gamma)$ is infinite.
        \begin{enumerate}
            \item $\Gamma$ contains a unique Coxeter SIL and no non-Coxeter SIL exactly when $\Out(W_\Gamma)$ is virtually $\Z$. 
            \item $\Gamma$ contains two or more SILs which are Coxeter SILs and do not form an FSIL or STIL exactly when $\Out(W_\Gamma)$ is virtually abelian but not virtually $\Z$.
            \item $\Gamma$ contains a non-Coxeter SIL, an FSIL, or a STIL exactly when $\Out(W_\Gamma)$ is large. 
         \end{enumerate}
    \end{enumerate}
\end{theorem}

\bibliographystyle{alpha}
\bibliography{Bibliography}

\end{document}